\documentclass[12pt]{amsart}
\usepackage{graphicx}
\usepackage{blkarray}
\usepackage{amsmath}
\usepackage{marginnote}
\usepackage[top=1cm, bottom=1cm, outer=1cm, inner=1cm, heightrounded, marginparwidth=1cm, marginparsep=1cm]{geometry}
\usepackage{hhline}
\usepackage{amsthm}
\usepackage{amssymb}
\usepackage{fullpage}
\usepackage{multirow}
\usepackage[all]{xy}

\usepackage{arydshln}
\usepackage{lscape}
\usepackage{rotating}
\usepackage{float}
\usepackage{bm}
\usepackage{float}
\usepackage[makeroom]{cancel}
\usepackage{nameref}
\usepackage{tikz}
\usepackage{setspace}
\usepackage{multirow}
\usepackage{appendix}
\usepackage[makeroom]{cancel}
\usepackage{color}
\usepackage{etex}
\usepackage{wrapfig}
\usepackage{fullpage}
\usepackage{verbatim}
\usepackage{tikz}
\usetikzlibrary{arrows,decorations.markings}
\usetikzlibrary{shapes}
\usetikzlibrary{matrix}
\usetikzlibrary{graphs}
\usepackage{arydshln}
\usetikzlibrary{backgrounds}
\usepackage{amsmath, amssymb}
\usepackage{fullpage}

\usepackage[makeroom]{cancel}
\usepackage{multirow}
\usepackage{tikz}
\usetikzlibrary{arrows,decorations.markings}
\usepackage[pdftex,
colorlinks,
linkcolor=blue,
urlcolor=blue,
citecolor=red
]{hyperref}
\usepackage{bookmark}
\newcommand{\teq}{\trianglelefteq}
\theoremstyle{plain}
\newtheorem{theorem}{Theorem}
\newtheorem{lemma}[theorem]{Lemma}

\newtheorem{prop}[theorem]{Proposition}
\newtheorem{cor}[theorem]{Corollary}

\newtheorem{conj}[theorem]{Conjecture}
\theoremstyle{definition}

\newtheorem{example}[theorem]{Example}

\theoremstyle{remark}
\newtheorem{remark}[theorem]{Remark}
\numberwithin{equation}{section}
\def\sub{\subseteq}

\def\F{\mathbb F}

\def\Z{\mathbb Z}
\def\bB{\mathbf B}
\def\bG{\mathbf G}

\def\bT{\mathbf T}
\def\bU{\mathbf U}

\def\cF{\mathcal F}
\def\cA{\mathcal A}
\def\cB{\mathcal B}
\def\cC{\mathcal C}
\def\cD{\mathcal D}

\def\cH{\mathcal H}
\def\cI{\mathcal I}
\def\cJ{\mathcal J}
\def\cK{\mathcal K}
\def\cL{\mathcal L}
\def\cN{\mathcal N}
\def\cP{\mathcal P}

\def\cS{\mathcal S}
\def\cT{\mathcal T}

\def\cZ{\mathcal Z}

\def\rA{\mathrm A}
\def\rB{\mathrm B}
\def\rC{\mathrm C}
\def\rD{\mathrm D}
\def\rE{\mathrm E}
\def\rF{\mathrm F}
\def\rG{\mathrm G}

\def\rU{\mathrm U}
\def\rY{\mathrm Y}
\def\fC{\mathfrak C}
\def\fO{\mathfrak O}

\def\tY{\tilde Y}

\def\ba{\bar a}
\def\ua{{\underline a}}
\def\ub{{\underline b}}
\def\uc{{\underline c}}

\def\eps{\epsilon}
\def\im{\operatorname{im}}
\def\Ind{\operatorname{Ind}}
\def\SO{\operatorname{SO}}
\def\Sp{\operatorname{Sp}}
\def\Stab{\operatorname{Stab}}

\def\Inf{\operatorname{Inf}}
\def\Irr{\operatorname{Irr}}
\def\rk{\operatorname{rk}}
\def\Tr{\operatorname{Tr}}
\def\ux{\underline{x}}
\def\ut{\underline{t}}
\def\us{\underline{s}}

\def\la{\langle}
\def\ra{\rangle}

\numberwithin{theorem}{section}

\title[
On characters of Sylow $p$-subgroups of finite Chevalley groups $G(p^f)$ for arbitrary primes
]{
On the characters of Sylow $p$-subgroups of finite Chevalley groups $G(p^f)$ for arbitrary primes
}

\author{
Tung Le, Kay Magaard and Alessandro Paolini
}

\address{T. L.: Department of Mathematics and Applied Mathematics, University of Pretoria, Pretoria 0002, South Africa}
\email{lttung96@yahoo.com}

\address{A. P.: FB Mathematik, TU Kaiserslautern, 67653 Kaiserslautern, Germany.} \email{paolini@mathematik.uni-kl.de}

\thanks{Date: \today. 
}

\begin{document}

\maketitle

\begin{abstract}

We develop in this work a method to parametrize the set $\Irr(U)$ of irreducible characters of a 
Sylow $p$-subgroup $U$ of a finite Chevalley group $G(p^f)$ which is valid for arbitrary primes $p$, in 
particular when $p$ is a very bad prime for $G$. 
As an application, we parametrize $\Irr(U)$ when $G=\rF_4(2^f)$. 
\end{abstract}

\section{Introduction}

The study of finite groups and their representations 
is a major research topic 
in the area of pure mathematics. 
An important open challenge is to determine the irreducible modular representations of finite simple groups. Particular focus has 
been dedicated to finite Chevalley groups. 

Let $q$ be a power of the prime $p$, and let $\F_q$ be the field with $q$ elements. 
Let $G$ be a finite Chevalley group defined over $\F_q$. 
For $H \le G$, denote by $\Irr(H)$ the set of ordinary irreducible characters of $H$. 
Due to the 
work of 
Lusztig, a great amount of information on $\Irr(G)$ has been determined, 
for instance irreducible character degrees and values of unipotent characters; see \cite{Car1} and \cite{DM}. 
The problem of studying modular representations of $G$ over a field of characteristic $\ell \ne p$ is still wide open. 

One of the approaches 
to this problem is to 
relate the modular representations of $G$ with the irreducible characters of a Sylow $p$-subgroup $U$ of $G$. 
Namely by inducing elements of $\Irr(U)$ to $G$ 
one gets $\ell$-projective characters, 
which yield approximations to the $\ell$-decomposition matrix of $G$. 
This is particularly important when $p$ is a bad prime for $G$, in that a definition of generalized Gelfand-Graev characters 
is yet to be formulated. 
Such an approach has proved to be successful in the cases of $\SO_7(q)$, $\Sp_6(q)$ \cite{HN14} and $\SO_8(q)$ \cite{Pao18}. In order to achieve this, obtaining a suitable parametrization of the set $\Irr(U)$ is an unavoidable step.

Another crucial motivation of this work originates from the following 
conjecture on finite groups of Lie type which has been suggested to us by G.~Malle.
The data for unipotent characters of $G$ in \cite[Chapter 13]{Car1} and 
those known for $\Irr(U)$ point out a \emph{strong} link between the rows of $\ell$-decomposition matrices of $G$, labelled by $\Irr(G)$, and their columns, labelled by suitable characters $\Ind_U^G(\chi)$ for $\chi \in \Irr(U)$.



\begin{conj}[Malle]\label{conj:Malle} Let $G$ be a finite Chevalley group defined over $\F_q$ with $q=p^f$ and $p$ a bad prime for $G$, and let $U$ be a Sylow $p$-subgroup of $G$. Then for every 
	cuspidal character $\rho \in \Irr(G)$, there 
	exists $\chi \in \Irr(U)$ such that $\chi(1)=\rho(1)_p$. 
\end{conj}

This conjecture is verified in the following cases: $\rB_2(2^f)$ \cite[\S7]{Lus03}, $\rG_2(p^f)$ for $p \in \{2, 3\}$ \cite[Section 3]{LMP18a}, 
$\rF_4(3^f)$ 
\cite[\S4.3]{GLMP16}, 
$\rD_i(2^f)$ for $i=4, 5, 6$ and $\rE_6(p^f)$ for $p \in \{2, 3\}$ \cite{LMP18b}, and $\rE_8(5^f)$ \cite{LM15}. Here, we confirm Conjecture \ref{conj:Malle} for $G=\rF_4(2^f)$. In particular, if $\rho \in \Irr(G)$ is one of the cuspidal characters $\rF_4^I[1]$ or $\rF_4^{II}[1]$ in the notation of \cite[\S13.9]{Car1}, then $\rho(1)_2=q^4/8$, and we do find irreducible characters of $U$ of degree $q^4/8$ in the family $\cF_{7, 2}^8$ in Table \ref{tab:parF4}.

We lay the groundwork for a package in GAP4 \cite{GAP4}, whose code is available at \cite{LMPdF4}, 
in order to build a database for the generic character table of $\rU\rF_4(2^f)$, in particular to find suitable replacements of generalized Gelfand-Graev characters as in \cite{Pao18}. Furthermore, we verify the generalization of Higman's conjecture in \cite{Hig60} for the group $\rU\rF_4(2^f)$, namely the number of its irreducible characters is a polynomial in $q=2^f$ with integral coefficients. 
\begin{theorem}\label{theo:Intro}
	Let $G=\rF_4(q)$ where $q=2^f$, and let $U$ be a Sylow $2$-subgroup of $G$. 
	Then each irreducible character of $U$ is completely parametrized 
	as an induction of a linear character of a certain determined subgroup of $U$. 
	In 
	particular, 
	we have 
	$$|\Irr(
	U)|=2q^8+4q^7+20q^6+46q^5-136q^4-16q^3+158q^2-94q+17.$$
\end{theorem}
In this work, we first develop a parametrization of $\Irr(U)$ by means of positive root sets of $G$, which is valid for \emph{arbitrary} primes. 
This procedure generalizes the one in \cite{GLMP16} and \cite{LMP18b}, which does not work for type $\rF_4$ when $p=2$. 
In general, if $p$ is a very bad prime for $G$ then 
we lose some structural information when passing from patterns to pattern groups. 
In fact, let $\Phi^+$ be the set of positive roots in $G$. 
The product of root subgroups indexed by a certain set $\cP \subseteq \Phi^+$ forms a group despite $\cP$ not being a pattern, see Example \ref{ex:pat}.

We generalize the definition of pattern and quattern groups (see \cite[\S2.3]{GLMP16}) 
for every prime by means of the Chevalley relations of $U$. 
Then 
every $\chi \in \Irr(U)$ is constructed as an inflated/induced character from a 
quattern group $V_\chi$ of $U$ which is uniquely determined by the algorithm in Section \ref{sec:adaalg}. 
In the case when $V_{\chi}$ 
is abelian, 
i.e. a so-called \emph{abelian core}, the character $\chi$ is directly parametrized. The focus of the rest of the work is then devoted to studying 
the nonabelian $V_{\chi}$'s, which we call \emph{nonabelian cores}. 
In order to determine $\Irr(V_{\chi})$, we generalize the technique used in \cite[\S4.2]{LMP18b}, by constructing a graph $\Gamma$ associated to $V_{\chi}$ as in Section \ref{sec:armleg}.
When the prime $p$ is very bad, the graph $\Gamma$ may have a vertex of valency $1$. 

We remark that $\rF_4(2^f)$ is the highest rank exceptional group at a very bad prime. 
Furthermore, the type $\rF_4$ is a \emph{good small} example to fulfill our algorithm for the determination of $\Irr(U)$ for all primes. 
A parametrization of $\Irr(\rU\rF_4(p^f))$ 
has now been determined for all primes $p$. Namely \cite[\S4.3]{GLMP16} settled the case when $p \ge 3$, and in this work we deal with the case $p=2$. 


We observe the following phenomenon which occurs just for $\rF_4(2^f)$ among all finite Chevalley groups of rank $4$ or less. The number of 
irreducible characters arising from a certain nonabelian core as in Table \ref{tab:fam2F4} \emph{cannot be expressed as a polynomial in $q=p^f$}.
%
In detail, the numbers of such characters do have 
polynomial expressions in $q=p^f$ whenever either $f=2k$ or $f=2k+1$, that is, $q\equiv 1 \pmod 3$ and $q\equiv -1 \pmod 3$ respectively. 
However, surprisingly, the expression for the number of irreducible characters of $U$ of a fixed degree is always a polynomial in $q$ with rational coefficients. 

The structure of this work is as follows. 
In Section \ref{sec:prelim} we recall notation and preliminary results on character theory of finite groups and Chevalley groups. In Section \ref{sec:genpat}, we give the definition of pattern and quattern groups valid for all primes. In Section \ref{sec:adaalg}, we generalize \cite[Algorithm 3.3]{GLMP16} to obtain all abelian and nonabelian cores. 
We discuss in Section \ref{sec:armleg} on the method to decompose nonabelian cores. Finally, 
we apply in Section \ref{sec:comput} the method previously developed to give 
a full parametrization of $\Irr(\rU\rF_4(2^f))$.

\vspace{1mm}

\noindent

\textbf{Authorship}: The second author passed away on July 26th, 2018, shortly after the main results of the paper have been jointly obtained. 
The other authors of this work are deeply grateful to him 
for an inestimable collaboration experience and for all the insight that originated from him in detecting the distinguished behavior of $\rU\rF_4(2^f)$ and many related results. 

\vspace{1mm}

\noindent

\textbf{Acknowledgment:} Part of the work has been developed during visits: of the second author at the University of Pretoria and at the University of KwaZulu-Natal in June 2018, supported by CoE-Mass FA2018/RT18ALG/007; of the second author in June 2018 and of the first author in January 2019 at the Technische Universit\"at Kaiserslautern, supported by the SFB-TRR 195 ``Symbolic Tools in Mathematics and their Application" of the German Research Foundation (DFG) and NRF Incentive Grant; and of the third author in September and October 2018 at the Hausdorff Institute of Mathematics in Bonn during the semester ``Logic and Algorithms in Group Theory", supported by a HIM Research Fellowship. The third author acknowledges financial support from the SFB-TRR 195. 
We would like to thank E.~O'Brien for the computation of $\Irr(\rU\rF_4(2^f))$ for $f \in \{1, 2, 3\}$. We are grateful to
G.~Malle for his continuous support, comments and discussions on our project.















\section{Preliminaries}\label{sec:prelim}

We present in this section some 
definitions and well-known results on the character theory of finite groups and on the theory of finite Chevalley groups. 

We consider in this work only complex characters. Notation and fundamental results are taken from 
\cite{Is}. Let $G$ be a finite group. We denote by $\Irr(G)$ the set of irreducible characters of $G$. The centre and the kernel of the character $\chi \in \Irr(G)$ are denoted by $Z(\chi)$ and $\ker(\chi)$ respectively. For $\varphi \in \Irr(G/N)$ with $N \teq G$, we denote by $\Inf_{G/N}^G(\varphi)$ 
the inflation of $\varphi$ to $G$. For $H \le G$, we denote by $\chi|_H$ the restriction of a character $\chi \in \Irr(G)$ to $H$, and by $\Ind_H^G(\psi)$ or $\psi^G$ the induction of a character $\psi$ of $H$ 
to $G$. Moreover, we define 
$$
\Irr(G \mid \psi):=\{\chi \in \Irr(G) \mid \langle\chi|_H, \psi\rangle \ne 0\}=\{\chi \in \Irr(G) \mid \langle\chi, \psi^G\rangle \ne 0\}.
$$

For $N \teq G$, $\varphi \in \Irr(N)$ and $g, x \in G$, we denote by $g^x$ the element $x^{-1}gx$, and by ${}^x\varphi$ the element of $\Irr(N)$ defined by $g \mapsto\varphi(g^x)$. This defines an action of $G$ on $\Irr(N)$. By \cite[\S2.1]{GLMP16}, if 
$Z$ is a subgroup of the centre $Z(G)$ of $G$ such that $Z \cap N =\{1\}$, then the inflation from $G/N$ to $G$ defines a bijection between the sets $\Irr(G/N \mid \lambda)$ and $\Irr(G \mid \Inf_Z^{ZN}(\lambda))$ for every $\lambda \in \Irr(ZN/N)$. 

The main references used for the basic notions of finite Chevalley groups are \cite{Car2}, \cite{DM} and \cite{MT}. Let $p$ be a prime, and let $\overline{\F}_p$ be an algebraically closed field of characteristic $p$. Fix a positive integer $f$, let $q:=p^f$, and let $F_q$ be the automorphism of $\overline{\F}_p$ defined by $x \mapsto x^q$. Then we denote by $\F_q$ the field with $q$ elements defined by $\F_q:=\{x \in \overline{\F}_p \mid F_q(x)=x\}$. 

Let $\bG$ be a simple linear algebraic group defined over $\overline{\F}_p$, and let $F$ be a standard Frobenius morphism of $\bG$ such that $(a_{ij})_{i, j} \mapsto (a_{ij}^q)_{i, j}$. 
A Chevalley group $G$ is the finite group defined as the set of fixed points of $\bG$ under $F$. From now on, we fix an $F$-stable maximal torus $\bT$ of $\bG$ and an $F$-stable Borel subgroup $\bB$ of $\bG$ containing $\bT$. Let $\bU$ be the unipotent radical of $\bB$. Then $\bB=\bU \rtimes \bT$, and correspondingly $B=U \rtimes T$, where $B,T$ and $U$ are the fixed points under $F$ of $\bB, \bT$ and $\bU$ respectively. If $G$ is a group of type $\rY$ and rank $r$, i.e. $G = \rY_r(q)$, then the group $U$ will be also denoted by $\rU\rY_r(q)$ in the sequel.

Let $\Phi$ be the root system of $\bG$ corresponding to the chosen $\bT$, and let $r$ be the rank of $\Phi$. 
Let $\{\alpha_1, \dots, \alpha_r\}$ be the subset of all positive simple roots of $\Phi$ with respect to the choice of $\bB$, whose enumeration agrees with 
the records of GAP4 \cite{GAP4}. Let $\Phi^+$ be the set of positive roots of $\Phi$, and $N:=|\Phi^+|$. 
Recall the partial order on $\Phi^+$, defined by $\alpha < \beta$ if and only if $\beta-\alpha$ is a positive combination of simple roots. We then choose an enumeration of the elements $\alpha_1, \dots, \alpha_N$ of $\Phi^+$, in such a way that 
$i<j$ whenever $\alpha_i<\alpha_j$, 
which also agrees with the enumeration in GAP4. 

For every $\alpha \in \Phi^+$ there exists a monomorphism $x_{\alpha}:\overline{\F}_p \to \bU$ satisfying the so-called Chevalley relations (see \cite[Theorem 1.12.1]{GLS3}). We denote by $\bU_{\alpha}$ (resp. $X_{\alpha}$) the root subgroup of $\bU$ (resp. $U$)  corresponding to $\alpha$, defined as the image under $x_{\alpha}$ of $\overline{\F}_p$ (resp. $\F_q$). For $1 \le i \le N$, we usually write $x_i$ and $X_i$ in place of $x_{\alpha_i}$ and $X_{\alpha_i}$ respectively. For every $s, t \in \F_q$ and $\alpha, \beta \in \Phi^+$, we recall the 
Chevalley commutator relation 
\begin{equation}\label{eq:CheFor}
[x_{\alpha}(t), x_{\beta}(s)]=
\prod_{\substack{i, j \in \mathbb{Z}_{>0} \mid i\alpha+j\beta \in \Phi^+}}
x_{i\alpha+j\beta}(c_{i,j}^{\alpha,\beta}(-s)^jt^i)
\end{equation}
where $c_{i,j}^{\alpha,\beta}$ are certain nonzero structure constants. 
In particular, $U$ is the product in any order of all its root subgroups. 

The prime $p$ is said to be \emph{very bad} for $G$ if it divides some 
$c_{i,j}^{\alpha,\beta}$. This happens if and only if 
$p=2$ in types $\rB_r$, $\rC_r$, $\rF_4$ and $\rG_2$ or $p=3$ in type $\rG_2$. In these cases, some $c_{i,j}^{\alpha,\beta}$ are actually equal to $\pm p$. In all other cases, we have $c_{i,j}^{\alpha,\beta} \in \{\pm 1\}$. The prime $p$ is called \emph{bad} for $G$ if it divides one of the coefficients 
in the decomposition of the highest root in $\Phi^+$ as a linear combination of $\alpha_1, \dots, \alpha_r$. 
A prime $p$ which is bad for $G$ is also very bad for $G$. 
The primes $p=2$ for types $\rD_r$ and $\rE_i$ with $i \in \{6, 7, 8\}$, $p=3$ for types $\rF_4$ and $\rE_i$ with $i \in \{6, 7, 8\}$, and $p=5$ for type $\rE_8$ are all the bad primes which are not very bad. A prime $p$ is called \emph{good} for $G$ when it is not a bad prime for $G$. 

Finally, we describe some properties of nontrivial irreducible characters of $\F_q$. Let $\Tr: \F_q \to \F_p$ 
be the field trace with respect to the extension $\F_q$ of $\F_p$. From now on, and for the rest of the work, we fix the irreducible character $\phi$ of $\F_q$ defined by $x \mapsto e^{2\pi i \Tr
(x)/p}$. Then $\ker(\phi)=\{t^p-t \mid t \in \F_q\}$. Every other nontrivial irreducible character of $\F_q$ is of the form $\phi \circ m_a$, where $a \in \F_q^\times$ and $m_a$ is the automorphism of $\F_q^\times$ defined by multiplication by $a$. It is easy to see that $\ker(\phi \circ m_a)=a^{-1}\ker(\phi)$.

%
%
%
%
%

\section{Pattern and quattern groups for all primes}
\label{sec:genpat}

In \cite[Section 3]{HLM16} and \cite[\S2.3]{GLMP16}, the notion of patterns and quatterns, defined when $p$ is not a very bad prime for $G$, are of major importance for the development of the methods for parametrizing $\Irr(U)$. We now define the following generalization of pattern groups and normal pattern groups for arbitrary primes. 

Let $\cP=\{\alpha_{i_1}, \dots, \alpha_{i_m}\}$ with $1 \le i_1 < \dots < i_m \le N$ be a subset of $\Phi^+$. We define 
$$
X_{\cP}:=\{x_{i_1}(t_{i_1})\cdots
x_{i_m}(t_{i_m}) \mid t_{i_1}, \dots t_{i_m} 
\in \F_q
\}.
$$
Considering each element in $X_\cP$ as an $m$-tuple, we have $|X_\cP|=q^{|\cP|}$. Further, if $X_\cP$ is a group then it is independent on the order of the $\alpha_{i_k}$'s in $\cP$. Here the $\alpha_{i_k}$'s are usually ordered by increasing indices if not specified otherwise. 
So the set $X_{\cP}$ is well-defined under this order setup.
We are mostly interested in those subsets $\cP$ of $\Phi^+$ such that $X_{\cP}$ is a group. 
\begin{prop}\label{prop:Pgp}
Let $\cP=\{\alpha_{i_1}, \dots, \alpha_{i_m}\} \subseteq \Phi^+$ with $1 \le i_1 < \dots < i_m \le N$. Then $X_{\cP}$ is a group 
if and only if for every $1 \le j< k \le m$ and every $s_{i_j}, s_{i_k} \in \F_q$ we have 
$$
[x_{i_j}(s_{i_j}), x_{i_k}(s_{i_k})] \in X_{\cP}.
$$
\end{prop}
\begin{proof}
For this proof, we write $X_{\cP}$ as $X_{\cP_m}$ and arrange positive roots in $\cP$ in decreasing order. It suffices to prove the converse of the above statement by induction on $m$. Let us assume that 
$[x_{i_j}(s_{i_j}), x_{i_k}(s_{i_k})]\in X_{\cP_m}$ for each $1 \le j< k \le m$ and $s_{i_j}, s_{i_k} \in \F_q$. Recall that $X_{\alpha_i}$ is a group itself for all $i$. The claim clearly holds for $m=1$.
For some $m>1$, it is enough to show that $xy \in X_{\cP_m}$ for all $x, y \in X_{\cP_m}$. Write $x=ax'$ and $y=by'$ for some $a,b\in X_{\cP_{m-1}}$ and $x', y' \in X_{\alpha_{i_m}}$. We have 
$$
xy=ax'by'=ax'bx'^{-1}b^{-1}bx'
=
a[x'^{-1}, b^{-1}]bx'y'.
$$
By the decreasing order of positive roots in $\cP$, we notice that $[x_{\alpha_{i_m}},x_{\alpha_i}]\in X_{\cP_{m-1}}$ for all $i$. Due to the formula $[r,st]=[r,t][r,s]^t$, we have $[x',b]\in X_{\cP_{m-1}}$ by induction hypothesis. Thus, $xy=(a[x'^{-1}, b^{-1}]b)(x'y')\in X_{\cP_{m-1}}X_{\alpha_{i_m}}=X_{\cP_m}$.
%
\end{proof}



If the conditions of Proposition \ref{prop:Pgp} are satisfied, we say that $X_{\cP}$ is a \emph{pattern group}. 

\begin{example}\label{ex:pat}
Consider $\rU\rB_2(q)$, and let $\alpha_1$ (resp. $\alpha_2$) be its long (resp. short) simple root. Let $\cP:=\{\alpha_2, \alpha_1+\alpha_2\}$. Then $X_\cP$ is a pattern 
group if and only if $p = 2$. 
Notice that $\cP$ is \emph{not} a pattern in the sense of \cite[Definition 3.1]{HLM16}. 

%
%
\end{example}

We would like to have a notion of normality of 
pattern groups. By applying the same method as in the proof of Proposition \ref{prop:Pgp}, it is straightforward to prove the following. 

\begin{prop}\label{prop:norP}
Let $\cP=\{\alpha_{i_1}, \dots, \alpha_{i_m}\} \subseteq \Phi^+$ be such that $X_{\cP}$ is a pattern group. Assume $\cN=\{\alpha_{j_1}, \dots, \alpha_{j_n}\} \subseteq \cP$. Then 
$X_\cN$ is a normal subgroup of $X_{\cP}$ if and only if for every $1 \le k \le m$ and $1 \le \ell \le n$ and every $s_{i_k}, s_{j_\ell} \in \F_q$ we have 
$$
[x_{i_k}(s_{i_k}), x_{j_\ell}(s_{j_\ell})] \in X_{\cN}.
$$
\end{prop}

Under the assumptions of Proposition \ref{prop:norP}, 
we say that $X_{\cN}$ is a \emph{normal pattern subgroup} of the pattern group $X_{\cP}$. 

\begin{example}\label{ex:nor}
Consider $\rU\rG_2(q)$, and let $\alpha_1$ (resp. $\alpha_2$) be its long (resp. short) simple root. Set $\cP:=\Phi^+$ and $\cN:=\{\alpha_1+2\alpha_2\}$. Then 
$X_{\cN}$ is a normal pattern group of $X_{\cP}$ if 
and only if $p=3$ (see \cite[\S3]{LMP18a} for a full parametrization of $\rU\rG_2(3^f)$). 
Notice that $\cN$ is \emph{not} a normal pattern in $\Phi^+$ in the sense of \cite[Definition 3.2]{HLM16}. 

%
%
\end{example}

Pattern groups over bad primes can readily be determined by using GAP4 in terms of the behaviour of 
the positive roots. We highlight this in the following proposition. 

\begin{prop}\label{prop:GAP2}
Let $\cP, \cN \subseteq \Phi^+$, and assume $p=2$ is a very bad prime for $G$. 
\begin{itemize}
\item[1)] The set $X_{\cP}$ is a pattern group 
if and only if for every $\alpha, \beta \in \Phi^+$, we have that
$$\alpha+\beta \in \Phi^+ \text{ and }
\alpha-\beta \notin \Phi 
\Longrightarrow \alpha+\beta \in \cP.$$
\item[2)] Let $X_{\cP}$ be a pattern group. Then $X_{\cN}$ is a normal pattern subgroup of $X_{\cP}$ 
if and only if for every $\alpha\in \cP$ and $\delta\in \cN$, we have that  
$$\alpha+\delta \in \Phi^+ \text{ and }
\alpha-\delta \notin \Phi 
\Longrightarrow \alpha+\delta \in \cN.$$
\end{itemize}
\end{prop}
\begin{proof}
This comes directly from Equation \eqref{eq:CheFor} and the fact that all the structure constants $c_{1,2}^{\alpha,\beta}$ and $c_{2,1}^{\alpha,\beta}$ vanish for every $\alpha,\beta\in\Phi^+$ when $p=2$ (see \cite[Chapter 4]{Car2}).
\end{proof}

Let $X_{\cP}$ be a pattern group, and let 
$X_{\cN}$ be a normal pattern subgroup of $X_{\cP}$. 
If $\cS=\cP \setminus \cN$, we put 
$X_{\cS}:=X_\cP/X_\cN$ and we call it the corresponding \emph{quattern group}. Given such a set $\cS$, define 
$$
\cZ(X_\cS):=\{\gamma \in \cS \mid 
X_{\gamma} \subseteq Z(X_\cS)
\},
$$
and
$$
\cD(X_\cS):=\{\gamma \in \cS \mid 
X_\cS=
X_{\gamma} \times 
X_{\cS \setminus \{\gamma\}}
\}. 
$$
For fixed $\cZ\subseteq \cZ(X_\cS)$, define
$$
\Irr(X_\cS)_\cZ:=\{
\chi \in \Irr(X_\cS) \mid 
\chi|_{X_\gamma} \ne 1_{X_\gamma} \text{ for every } 
\gamma \in \cZ
\}.
$$
We recall the following formula 
(see \cite[Equation (3)]{LMP18b}), 
\begin{equation}\label{eq:allfam}
\sum_{\chi \in \Irr(X_\cS)_\cZ}\chi(1)^2=q^{|\cS\setminus \cZ|}(q-1)^{|\cZ|}. 
\end{equation}


Fixed a character $\chi \in \Irr(U)$, define 
$$\rk(\chi):=\{\alpha \in \Phi^+ \mid X_\alpha \subseteq \ker \chi\}.$$

\begin{prop}
\label{prop:Exist}
Let $\chi \in \Irr(U)$. Then $X_{\rk(\chi)}$ is a normal pattern subgroup of $U$. 

Conversely, let $\cN \subseteq \Phi^+$ be such that 
$X_{\cN}$ is a normal pattern subgroup in $U$. Then 
there exists $\chi_\cN \in \Irr(U)$ such that 
$\rk(\chi_\cN)=\cN$. 
\end{prop}

\begin{proof} 

Let $\chi\in \Irr(U)$. For every $\alpha\in \rk(\chi)$ and $\beta\in\Phi^+$, we claim that $[x_\alpha(t),x_\beta(s)]\in X_{\rk(\chi)}$ for all $s,t\in\F_q$; this of course would imply $x_\alpha(t)^{x_\beta(s)} \in X_{\rk(\chi)}$. 
We have $\la \alpha,\beta\ra\in \{\rA_1{\times}\rA_1, \rA_2, \rB_2, \rG_2\}$. Due to the properties of $\rk(\chi)$, it suffices to prove that for the cases $\la \alpha,\beta\ra \in \{\rB_2, \rG_2\}$ the claim is true for all irreducible constituents of $\chi|_{\la X_\alpha,X_\beta \ra}$. 

By \cite[\S3]{LMP18a}, the first statement holds for $G$ of type $\rG_2$ at any prime. 
From the fact that $\rU\rB_2(q)\cong \rU\rG_2(q)/X_{3\alpha_1+2\alpha_2}X_{3\alpha_1+\alpha_2}$ and that $\Irr(\rU\rG_2(q))$ is partitioned by positive root sets as in \cite[\S3]{LMP18a}, if $\alpha\in\rk(\chi)$ then both $X_{\alpha+\beta}$ and $X_{i\alpha+j\beta}$ are contained in $\ker(\chi)$; thus, the claim follows.

For the converse, let $\cN \subseteq \Phi^+$ be such that 
$X_{\cN}$ is a normal pattern group in $U$. Set $X:=U/X_{\cN}=X_{\Phi^+ \setminus \cN}$, and by slight abuse of notation write $X_\alpha$ instead of $X_\alpha X_{\cN}$ for every root subgroup $X_\alpha$. Let $\lambda\in\Irr(\cZ(X))$ be such that $\lambda|_{X_\alpha}\neq 1_{X_\alpha}$ for all $X_\alpha\leq Z(X)$. By properties of induction, for every constituent $\chi\in \Irr(X \mid \lambda)$ we have $\chi|_{X_\alpha}=\chi(1)\lambda|_{X_\alpha}$ for all $X_\alpha\leq Z(X)$. Thus, $\rk(\chi)=\emptyset$. So the inflation $\chi_\cN$ of $\chi$ to $U$ satisfies $\rk(\chi_\cN)=\cN$.
\end{proof}





We now determine a partition of $\Irr(U)$ in terms of 
the so-called representable sets. 
We call $\Sigma \subseteq \Phi^+$ a \emph{representable set} if $\Sigma=\cZ(U/X_{\cN})$ for some $\cN \subseteq \Phi^+$ such that 
$X_\cN \teq U$. 
Notice 
that if $X_{\cN_i} \teq U$ and $\Sigma_i:=\cZ(U/X_{\cN_i})$ 
for $i=1, 2$, then $\Sigma_1=\Sigma_2$ if and only if $\cN_1=\cN_2$. 
Hence given a representable set $\Sigma \subseteq \Phi^+$, we 
can define $\cN_{\Sigma}$ to be the unique set corresponding to a normal pattern group of $U$ such that $
\cZ(U/X_{\cN_{\Sigma}})
=\Sigma$. 
For a representable set $\Sigma$, denote
$$\Irr(U)_\Sigma:=\{\Inf_{U/X_{\cN_{\Sigma}}}^{U}(\chi) 
\mid \chi \in \Irr(U/X_{\cN_{\Sigma}})_{\Sigma}\}.$$

\begin{remark}
	When $p$ is not a very bad prime for $G$, then the definition of representable sets given in this work is consistent with \cite[Section 5]{HLM16}. 
\end{remark}

The desired partition follows by Proposition \ref{prop:Exist} and the uniqueness of $\rk(\chi)$ for every $\chi\in\Irr(U)$. 

\begin{prop}
We have that 
$$
\Irr(U)=\bigsqcup_{\Sigma \subseteq \Phi^+ \mid \Sigma \text{ representable }}
\Irr(U)_{\Sigma}.
$$
\end{prop}


Finally, we remark that all representable sets in low rank are determined by computer algebra. Namely these are in bijection with normal pattern groups in $\Phi^+$, and Proposition \ref{prop:GAP2} gives a criterion to check whether a subset of $\Phi^+$ 
gives rise to a normal pattern group in $U$. In this way, it is immediate to produce an efficient algorithm in GAP4 whose input is a record of the Chevalley relations and that 
gives all representable sets; see the function\verb| repSetAll| in our GAP4 code in \cite{LMPdF4}. 

\begin{center}
	\begin{table}[ht!]  
		\begin{tabular}{|c||c|c|c|c|c|c|}
			\hline
			Type & $\rA_4$ & \multicolumn{2}{c|}{$\rB_4/\rC_4$} & $\rD_4$& \multicolumn{2}{c|}{$\rF_4$} \\
			\hline
			Prime & any & $p=2$ & $p \ge 3$ & any & $p=2$ & $p \ge 3$ \\
			\hline
			$\#$Rep. sets & $42$ & $98$ & $70$ & $50$ & $190$ & $105$\\
			\hline
		\end{tabular}
	\medskip
	\caption{The number of representable sets in rank $4$ at different primes. }
	\label{tab:repsets} 
	\end{table}
\end{center}
\vspace{-1cm}

We collect the numbers of representable sets in rank $4$ in Table \ref{tab:repsets}. Notice that these numbers are the same as in \cite[Table 2]{HLM16} when $p$ is not a very bad prime for $G$, namely they coincide with the numbers of antichains in $\Phi^+$. On the other hand, fixed a type in Table \ref{tab:repsets} for which $2$ is a very bad prime, we see that the number of representable sets for $p=2$ and for $p \ge 3$ is considerably different. 

\section{Reduction algorithm}\label{sec:adaalg}


We develop in this section a reduction algorithm for the study of the sets 
$\Irr(U)_{\Sigma}$ with $\Sigma \subseteq \Phi^+$ representable, which is an adaptation of \cite[Algorithm 3.3]{GLMP16}. Namely we establish 
 a bijection between a set of the form $\Irr(X_\cS)_{\cZ}$, with 
 $\cZ \subseteq \cZ(X_\cS)$, and a set $\Irr(X_{\cS'})_{\cZ'}$, with 
 $\cZ' \subseteq \cZ(X_{\cS'})$, 
 where $|\cS'|\lneq |\cS|$. More precisely, our goal is to develop an algorithm which takes $\Sigma$ as input, and outputs the following decomposition,
$$
\Irr(U)_{\Sigma} \longleftrightarrow \bigsqcup_{\fC \in \fO_1}\Irr(X_{\cS})_{\cZ} \sqcup \bigsqcup_{\fC \in \fO_2}\Irr(X_{\cS})_{\cZ},
$$
where each of the sets $\fC$ is a tuple $(\cS, \cZ, \cA, \cL, \cK)$ of 
positive roots, and the sets $\fO_1$ and $\fO_2$ are a measure for 
the complication of the parametrization of the characters of 
$\Irr(X_{\cS})_{\cZ}$. 

The set $\fO_1$ contains all families of characters whose parametrization is immediately provided by the algorithm. 
The remaining families 
in $\fO_2$, whose study requires more work, almost always highlight a pathology 
of the group $U$ at very bad primes. For example, they often contain characters whose degree is not a power of $q$. 
The families in $\fO_2$ 
shall be in turn reduced to few enough cases to be studied in an ad-hoc way.

We introduce the following notation, in a similar way as in \cite[\S2.3]{GLMP16}. Assume that $X_{\cS_i}$ is a quattern group with respect to $\cP_i$ and $\cN_i$ for $i=1, 2$. If $\cP_1=\cP_2$ and $\cN_1 \supseteq \cN_2$, for $\cL:=\cN_1 \setminus \cN_2$ we define $\Inf_\cL$ to be the inflation from $X_{\cS_1}$ to $X_{\cS_2}$, and if $\cL=\{\alpha\}$ we put $\Inf_\alpha:=\Inf_\cL$. If $\cN_1=\cN_2$ and $\cP_1 \subseteq \cP_2$, for $\cT:=\cP_2 \setminus \cP_1$ we define $\Ind_\cT$ to be the induction from $X_{\cS_1}$ to $X_{\cS_2}$, and if $\cT=\{\alpha\}$ we put $\Ind^\alpha:=\Ind^\cT$. 

We need the following adaptation of \cite[Lemma 3.1]{GLMP16}. The proof repeats mutatis mutandis. 

\begin{prop}\label{prop:small} Let $\cS=\cP \setminus \cK$ be such that $X_\cS$ is a quattern group, and let $\cZ \sub \cZ(X_\cS)$. Suppose that there exist $\gamma \in \cZ$ and $\delta, \beta \in \cS \setminus \{\gamma\}$ satisfying: 
\begin{itemize}
\item[1)] $
[x_{\beta}(s), x_{\delta}(t)]=x_\gamma(cs^it^j)
$
for some $c \ne 0$ and $i, j \in \Z_{\ge_1}$,
\item[2)] 
$[X_\alpha, X_{\alpha'}] \cap X_\beta = 1$ for all $\alpha, \alpha' \in \cS$, and 
\item[3)] 
$[X_\alpha, X_{\delta}]=1$ for every $\alpha \in \cS \setminus \{\beta\}$.
\end{itemize}
Define $\cP':=\cP \setminus \{\beta\}$, $\cK':=\cK \cup \{\delta \}$ and $\cS':=\cP'\setminus\cK'$. Then $X_{\cP'}$ is a pattern group and $X_{\cK'} \teq X_{\cP'}$, i.e. $X_{\cS'}$ is a quattern group. 
Moreover, we have a bijection
\begin{align*}
\Irr(X_{\cS'})_\cZ &  \to \Irr(X_\cS)_\cZ \\
\chi & \mapsto \Ind^\beta \Inf_\delta \chi
\end{align*}
by inflating over $X_{\delta}$ and inducing to $X_\cS$ over $X_{\beta}$.
\end{prop}

We now proceed to illustrate the adaptation of \cite[Algorithm 3.3]{GLMP16}. At each step, we assume that the tuple $\fC=(\cS, \cZ, \cA, \cL, \cK)$ is constructed and currently taken into consideration.

\begin{itemize}
\item[\textbf{Input.}] Our input is a representable set $\Sigma$. We initialize $\fC$ by putting $\cS=\Phi^+ \setminus \cN_{\Sigma}$, $\cZ=\cZ(X_{\cS})$, and $\cA=\cL=\cK=\emptyset$ 
and $\fO_1=\fO_2=\emptyset$. 

\item[\textbf{Step 1.}] 
If $\cS=\cZ(X_\cS)$, then $X_{\cS}$ is abelian and the set $\Irr(X_{\cS})_{\cZ}$ is parameterized as 
$$\Irr(X_{\cS})_{\cZ}=\left\{\Ind^\cA \Inf_\cK \lambda_\ub^\ua\mid \ua \in (\F_q^\times)^{|\cZ|}, \ub \in (\F_q^\times)^{|\cS\setminus\cZ|}\right\},$$
where $\lambda_\ub^\ua$ is a linear character of $\Irr(X_{\cS})$ supported on $X_{\cZ}$ as 
in \cite[Lemma 3.5]{GLMP16} and \cite[\S2.4]{LMP18b}. The family arising from $\fC$ is therefore completely parametrized and we add the 
element $\fC$ to $\fO_1$. 


\item[\textbf{Step 2.}] 
If $\cS \ne \cZ(X_\cS)$ and at least one pair of positive roots satisfies the assumptions of Proposition 
\ref{prop:small}, then we choose from those the pair $(\beta, \delta)$ to be the unique pair having minimal $\beta$ among the ones having maximal $\delta$ (with respect to the linear ordering on $\Phi^+$), and we put 
$\fC':=(\cS', \cZ, \cA', \cL', \cK')$, with 
$$\cS'=\cS \setminus \{\beta, \delta\}, \quad \cA'=\cA \cup \{\beta\}, \quad \cL'=\cL \cup \{ \delta\}, \quad \text{ and }\quad
\cK'= \cK \cup \{\delta\}.$$
Then by Proposition \ref{prop:small}, we have a bijection 
$$\Ind^\beta \Inf_\delta: \Irr(X_{\cS'})_\cZ \longrightarrow \Irr(X_\cS)_\cZ.$$
We go back to Step 1 with $\fC'$ in place of $\fC$.

\item[\textbf{Step 3.}] If $\cS \ne \cZ(X_\cS)$, if 
no pair of positive roots satisfies Proposition \ref{prop:small}, 
and if $\cZ(X_\cS) \setminus (\cZ \cup \cD(X_\cS)) \ne \emptyset$, 
then in a similar way as \cite[\S2.4]{LMP18b}, we choose the maximal element 
$\gamma$ in $\cZ(X_\cS) \setminus (\cZ \cup \cD(X_\cS))$, and we put 
$$\cS'=\cS \setminus \{\gamma\}, \qquad \cK'= \cK \cup \{\gamma\} \qquad \text{ and }\qquad \cZ''=\cZ\cup \{\gamma\}
.$$ 
Then we have that 
$$\Irr(X_{\cS})_{\cZ}=\Irr(X_{\cS'})_{\cZ} \sqcup \Irr(X_{\cS})_{\cZ''},$$
and we go back to Step 1 carrying each of 
$$\fC':=(\cS', \cZ, \cA, \cL, \cK') \qquad \text{and} \qquad \fC'':=(\cS , \cZ'', \cA, \cL, \cK ).$$


\item[\textbf{Step 4.}] 
If the tuple $\fC$ is such that $\cS \ne \cZ(X_\cS)$, 
no pair of positive roots satisfies Proposition \ref{prop:small} 
and $\cZ(X_\cS) \setminus (\cZ \cup \cD(X_\cS)) = \emptyset$, 
then we add $\fC$ to $\fO_2$. 


\end{itemize}
%
%
%
%

The elements of the form $\fC$ of the set $\fO_1$ (resp. $\fO_2$) 
are called the \emph{abelian} (resp. \emph{nonabelian}) \emph{cores} of $U$, 
as the corresponding groups $X_\cS$ are abelian (resp. nonabelian). 
As in \cite{LMP18b}, we sometimes write $(\cS, \cZ)$ in short for the 
core $\fC=(\cS, \cZ, \cA, \cL, \cK)$. 

In a similar way as \cite[\S2.4]{LMP18b}, given a nonabelian core $\fC$ as above we identify, by slight abuse of notation, the sets $\cS$ and $\cZ$ with $\cS \setminus \cD(X_{\cS})$ and $\cZ \setminus \cD(X_{\cS})$ respectively. Namely we have that $X_{\cS}=X_{\cS \setminus \cD(X_{\cS})} \times X_{\cD(X_{\cS})}$, hence 
$$
\Irr(X_{\cS})_{\cZ}=\Irr(X_{\cS \setminus \cD(X_{\cS})})_{\cZ \setminus \cD(X_{\cS})} \times 
\Irr(X_{\cD(X_{\cS})})_{\cZ \cap \cD(X_{\cS})},
$$
and the set $\Irr(X_{\cD(X_{\cS})})_{\cZ \cap \cD(X_{\cS})}$ is readily parametrized since $X_{\cD(X_{\cS})}$ is abelian. 


We say that a nonabelian core $\fC$ corresponding to $\cS$ and $\cZ$ is a \emph{$[z, m, c]$-core} if 
\begin{itemize}
	\item $|\cZ|=z$, 
	\item $|\cS|=m$, and
	\item there are exactly $c$ pairs $(i, j)$ with $i<j$ corresponding to nontrivial Chevalley relations in $X_{\cS}$, i.e. such that $x_i(s), x_j(t) \in X_{\cS}$ for all $s, t \in \F_q$ and $[x_i(s), x_j(t)] \ne 1$. 
\end{itemize}
We also say that the triple $[z, m, c]$ is the \emph{form} of the core $\fC$. Recall that nonabelian core forms can be easily read off from the output of the algorithm described above and implemented in GAP4. 

We finish by recalling the forms in groups of rank $4$. 
For type $\rA_4$ there are no nonabelian cores at any prime. As in \cite[Section 4]{GLMP16}, there is just one $[3, 10, 9]$-core in type $\rB_4$ for $p \ge 3$ and in type $\rD_4$ for arbitrary primes, there are $6$ nonabelian cores of different forms in type $\rF_4$ for $p \ge 3$, and there are no nonabelian cores in type $\rC_4$ for $p \ge 3$. In the case of $\rU\rB_4(2^f) \cong \rU\rC_4(2^f)$ we have $51$ nonabelian cores of the form $[2, 4, 1]$, one $[4, 8, 2]$-core and one $[4, 11, 6]$-core. Finally, we collect in the first two columns of Table \ref{tab:parF4} the $11$ triples $[z, m, c]$ giving rise to a nonabelian core of $\rU\rF_4(2^f)$ and the number of cores of a fixed form. 

\begin{remark} \label{rem:241}
	In contrast with \cite[Theorem 4]{LMP18b}, we have that two cores with the same form $[z, m, c]$ are \emph{not} necessarily isomorphic. Namely we see from Table \ref{tab:parF4} that in $\rU\rF_4(2^f)$ the cores of the form $[4, 8, 4]$ split into at least two isomorphism classes, since the sets of the form $\Irr(X_{\cS})_{\cZ}$ are evidently different, and so do cores of the form $[4, 12, 9]$ and $[5, 9, 4]$. 
	
	On the other hand, it is easy to check that whenever $2$ is a very bad prime for $G$, any core of the form $[2, 4, 1]$ is isomorphic 
	to the $\rB_2$-core of the form $[2, 4, 1]$ corresponding to $\cS=\Phi^+=\{\alpha_1, \dots, \alpha_4\}$ and 
	$\cZ=\cZ(X_{\cS})=\{\alpha_3, \alpha_4\}$. Its study is well-known, see for instance \cite[\S7]{Lus03}. Thus $\rF_4(2^f)$ is the Chevalley group of minimum rank in whose Sylow $p$-subgroup we find non-isomorphic $[z, m, c]$-cores. 
\end{remark}




\section{Reducing nonabelian cores}\label{sec:armleg}

By virtue of Section \ref{sec:adaalg}, the focus from now on is on the study of the families $\Irr(X_\cS)_\cZ$ where $\fC=(\cS, \cZ)$ is 
a nonabelian core. 
Our methods will involve again 
inflation and induction from smaller subquotients. 
The groups involved in our procedure need not be root subgroups anymore. In particular, we need to deal with 
diagonal subgroups of products of root subgroups of $U$. In order to do this, 
we need the following result from \cite[\S4.1]{LMP18b}, which we 
recall here in a more compact form. 

\begin{prop}\label{prop:redlem}
Let $V$ be a finite group. Let $H \le V$, and let $X$ be a transversal 
of $H$ in $V$. Assume that there exist subgroups $Y$ and $Z$ of $H$, and $\lambda \in \Irr(Z)$, such that 
\begin{itemize}
\item[(i)] $Z$ is a central subgroup of $V$, 
\item[(ii)] $Y$ is a central subgroup of $H$,
\item[(iii)] $Z \cap Y=1$, 
\item[(iv)] $[X, Y] \subseteq Z$, and 
\item[(v)]$Y':=\Stab_Y(\lambda)$ has a complement $\tilde{Y}$ in $Y$. 
\end{itemize}
Let $X':=\Stab_X(\lambda)$, and let $H':=HX'$. Then $H'=\Stab_V(\Inf_{Z}^{Z\tilde{Y}}(\lambda))$ is a subgroup of $V$ such 
that $\tilde{Y}\ker(\lambda) \teq H'$, and we have a bijection 
\begin{equation*}
\Ind_{H'}^{V} \Inf_{H'/\tilde{Y}\ker(\lambda)}^{H'}: \Irr(H'/\tY\ker(\lambda) \mid \lambda) \longrightarrow \Irr\left(V \mid \lambda\right).
\end{equation*}
\end{prop}

Throughout the rest of the work, we 
keep the notation of Proposition \ref{prop:redlem} for the group $V$, its subquotients and $\lambda \in \Irr(Z)$ satisfying assumptions (i)--(v). These will be specified in each case taken into consideration. 
We use the terminology of \cite[Definition 10]{LMP18b} and we call $\tilde{X}$ and $\tilde{Y}$ an \emph{arm} and a \emph{leg} of $X_\cS$ respectively, and $X$ and $Y$ a \emph{candidate for an arm} and a \emph{candidate for a leg} in $X_\cS$ respectively. 

In the case when $V=X_\cS$, the check of the validity of the assumptions of Proposition \ref{prop:redlem} translates into a condition 
on the underlying set of positive roots involved, which can be 
carried out by computer investigation. In particular, \cite[Corollary 13]{LMP18b} generalizes in the following way when $p$ is a very bad prime for $G$. 

\begin{cor}\label{cor:plus}
	Let $\cS \subseteq \Phi^+$ be such that $X_\cS$ is a quattern group. 
	Assume that there exist subsets 
	$\cZ$, $\cI$ and $\cJ$ of $\cS$, such that 
	\begin{itemize}
		\item[(0)] $X_{\cS \setminus \cI}$ is a quattern group, 
		\item[(i)] $\cZ\subseteq \cZ(X_\cS)$, 
		\item[(ii)] $\cJ\subseteq \cZ(X_{\cS \setminus \cI})$, 
		\item[(iii)] $\cJ \cap \cZ = \emptyset$, and 
		\item[(iv)] if $\alpha \in \cI, \beta \in \cJ$ and $[X_\alpha, X_\beta] \ne 1$, then $[X_\alpha, X_\beta] \subseteq X_\cZ$. 
	\end{itemize}
	Let us put $Z=X_\cZ$, $X=X_\cI$, $Y=X_\cJ$ and $H=X_{\cS \setminus \cI}$. In the notation of Proposition \ref{prop:redlem}, we have a bijection 
	$$\Ind_{\tilde{H}}^{X_\cS} \Inf_{H'/\tilde{Y}\ker(\lambda)}^{H'}: \Irr(H'/\tilde{Y}\ker(\lambda) \mid \lambda) \longrightarrow \Irr(X_{\cS} \mid \lambda).$$
\end{cor}

Let $\fC=(\cS, \cZ)$ be a fixed nonabelian core. In order to find sets $\cI$ and $\cJ$ as in Corollary \ref{cor:plus}, 
we define the following generalization of the graph in 
\cite[\S4.2]{LMP18b}. Define a graph $\Gamma$ in the 
following way. The vertices are labelled by elements in $\cS$, 
and there is an edge between $\alpha$ and $\beta$ if and only if 
$1 \ne [x_\alpha(s), x_{\beta}(t)] \in X_\cZ$ 
for some $s, t \in \F_q$. 

We have the notion of connected components and circles in 
$\Gamma$ as in \cite[\S4.2]{LMP18b}. The \emph{heart} of 
$\Gamma$, which we usually denote by $\cH$, is the set of 
roots in $\cS$ whose corresponding vertex in $\Gamma$ has valency zero. We say that $\fC$ is a \emph{heartless core} if 
$\cH=\emptyset$. 

In Chevalley groups of rank $4$ or less, the shape of each connected component of $\Gamma$ with at least one edge is verified to be as follows. We have either a linear tree, or a union of circles together with possibly few subgraphs isomorphic to linear trees which share with it exactly a vertex (see the second graph in Figure \ref{fig:2circ}). In particular, the shape of the graph 
$\Gamma$ is different from the ones of the graphs obtained in \cite[\S4.2]{LMP18b} due to the existence of vertices of valency $1$; these correspond to roots which form a $\rB_2$--subsystem with its unique neighbor in $\Gamma$. Hence we need a new method to define the sets $\cI$ and $\cJ$. 

We assume, without loss of generality, that $\Gamma$ is a connected graph with at least one edge. 
We now construct \emph{uniquely defined} candidates for the 
sets $\cI$ and $\cJ$ in this case, such that $\cS \setminus \cH=\cI \cup \cJ$. The reason why such constructed $\cI$ and 
$\cJ$ are likely to satisfy the assumptions of Corollary \ref{cor:plus} lies in the fact that the induced graph 
$\Gamma|_\cI$ has no edges. That is, no elements of $\cI$ 
are connected to each other. In fact, as in \cite[Remark 14]{LMP18b}, if $\fC$ is a heartless core then such $\cI$ and 
$\cJ$ 
do indeed satisfy the conditions of Corollary \ref{cor:plus}. 

We recall the natural notion of a distance $d$ defined on the vertices of a linear tree $\Delta$. Let $\eps$ and $\delta$ be two vertices in $\Delta$. If $\eps=\delta$ then we put $d(\eps, \delta)=0$. Assume that $\eps \ne \delta$. Then we define $d(\eps, \delta)=s$ if and only if there exist $s$ edges $\{\beta_i, \beta_{i+1}\}$ for $i=1, \dots, s$, such that $\beta_1=\eps$, $\beta_{s+1}=\delta$, and $\beta_i \ne \beta_j$ if $i \ne j$.

The construction of $\cI$ and $\cJ$ is as follows. We first assume that $\Gamma$ is a linear tree with set of vertices $V$. 
\begin{itemize}
	\item Let $\delta$ be the maximal root in $\Gamma$ with respect to the previously fixed linear ordering of $\Phi^+$. Then we set $\cJ_0:=\{\delta\}$, and for each $k \ge 1$ we define 
	$$
	\cI_k:=\{\beta \in V \mid d(\beta, \delta)=2k-1\} 
	\qquad \text{and} \qquad 
	\cJ_k:=\{\beta \in V \mid d(\beta, \delta)=2k\}.
	$$
	Finally, 
	we define 
	$$
	\cI:=\bigcup_{k} \cI_k 
	\qquad \text{and} \qquad 
	\cJ:=\bigcup_{k} \cJ_k.
	$$
\end{itemize}
We now assume that the union $\cC(\Gamma)$ of all 
circles in $\Gamma$ is nonempty. 


\begin{itemize}
\item 
We first follow the same procedure as \cite[\S4.2]{LMP18b}, namely we 
suitably enumerate the distinct circles $\cC_1$, $\dots$, $\cC_t$ of $\Gamma$ and we construct the sets $\cI_1, \dots, \cI_t$ and $\cJ_1, \dots, \cJ_t$ accordingly, such that 
$\cI_t \cup \cJ_t=\cC(\Gamma)$. 

\item Let $\cT$ be the set of subgraphs attached to $\cC(\Gamma)$ and isomorphic to linear trees. As previously remarked, if $\Delta \in \cT$ then $\Delta$ and $\cC(\Gamma)$ share a unique vertex, say $\delta$. Let $V$ be the set of vertices of $\Delta$. If $\delta \in \cJ_t$, then we set $\cJ_0(\Delta):=\{\delta\}$, and for each $k \ge 1$ we define 
$$
\cI_k(\Delta):=\{\beta \in V \mid d(\beta, \delta)=2k-1\} 
\quad \text{and} \quad 
\cJ_k(\Delta):=\{\beta \in V \mid d(\beta, \delta)=2k\}.
$$
Otherwise, we have that $\delta \in \cI_t$ since $\cI_t \cup \cJ_t=\cC(\Gamma)$. In this case, we set $\cI_0(\Delta):=\{\delta\}$, and for every $k \ge 1$ we define 
$$
\cI_k(\Delta):=\{\beta \in V \mid d(\beta, \delta)=2k\} 
\quad \text{and} \quad 
\cJ_k(\Delta):=\{\beta \in V \mid d(\beta, \delta)=2k-1\}.
$$
We then put
$$
\cI(\Delta) :=\bigcup_{k} \cI_k(\Delta) 
\qquad \text{and} \qquad 
\cJ(\Delta) :=\bigcup_{k} \cJ_k(\Delta) .
$$

\item Finally, we define 
$$
\cI:=\cI_t \cup \bigcup_{\Delta \in \cT} \cI(\Delta)
\qquad \text{and} \qquad 
\cJ:=\cJ_t \cup \bigcup_{\Delta \in \cT} \cJ(\Delta). 
$$


\end{itemize}

The general ideas of the construction just outlined are summarized in the two examples in Figure \ref{fig:2circ} which relate to the families $\cF_{7, 1}$ and $\cF_{10}$ of $\Irr(\rU\rF_4(2^f))$ in Table \ref{tab:parF4}. 

\begin{figure}[h] 
	\begin{center}
		\begin{minipage}{0.5\textwidth}
			\begin{center}
				\begin{tikzpicture}[scale=0.75, crc/.style={circle,draw=black!50,thick,
					inner sep=0pt,minimum size=8mm}, td/.style={rectangle,draw=black,thick, dotted,
					inner sep=0pt,minimum size=8mm},  rrt/.style={circle,draw=red!50,thick,
					inner sep=0pt,minimum size=8mm}, brt/.style={rectangle,draw=black,thick,
					inner sep=0pt,minimum size=8mm}, transform shape]
				\node (a) at (30:2) [brt]  {$\alpha_{13}$};
				\node (b) at (90:2) [td]  {$\alpha_1$};
				\node (c) at (150:2)   [brt] {$\alpha_{10}$};
				\node (d) at (210:2) [td]    {$\alpha_{11}$};
				\node (e) at (270:2)  [brt]  {$\alpha_4$};
				\node (f) at (330:2) [td]   {$\alpha_8$};
				\node (g) at (3.73, 1) [td]   {$\alpha_5$};
				\node (h) at (3.73, -1) [brt]   {$\alpha_7$};
				\draw (a) -- (b) -- (c) -- (d) -- (e) -- (f) -- (a) ;
				\draw (a) -- (g) -- (h) -- (f) ;
				\end{tikzpicture}
			\end{center}
		\end{minipage}\vline\begin{minipage}{0.5\textwidth}
			\begin{center}
				\begin{tikzpicture}[scale=0.75, crc/.style={circle,draw=black!50,thick,
					inner sep=0pt,minimum size=8mm}, td/.style={rectangle,draw=black,thick, dotted,
					inner sep=0pt,minimum size=8mm},  rrt/.style={circle,draw=red!50,thick,
					inner sep=0pt,minimum size=8mm}, brt/.style={rectangle,draw=black,thick,
					inner sep=0pt,minimum size=8mm}, transform shape]
				\node (a) at (45:1.414) [brt]  {$\alpha_4$};
				\node (b) at (135:1.414) [td]  {$\alpha_3$};
				\node (c) at (225:1.414)   [brt] {$\alpha_5$};
				\node (d) at (315:1.414) [td]    {$\alpha_9$};
				\node (e) at (45:2.828)  [td]  {$\alpha_6$};
				\node (f) at (45:4.242)  [brt]  {$\alpha_1$};
				\draw (a) -- (b) -- (c) -- (d) -- (a);
				\draw (a) -- (e) -- (f);
				\end{tikzpicture}
			\end{center}
		\end{minipage}
	\end{center}
	\caption{Two examples of graphs as described above in the case $G=\rF_4(2^f)$. On the left, $\cS$ and $\cZ$ correspond to the $[4, 12, 9]$-core associated to $\cF_{7, 1}$. On the right, $\cS$ and $\cZ$ are taken with respect to the only $[5, 11, 6]$-core; here $|\cT|=1$, and 
	$\Delta \in \cT$ is the graph with edges $\{\alpha_4, \alpha_6\}$ and $\{\alpha_6, \alpha_1\}$. The vertices in $\cJ$ (resp. $\cI$) are those surrounded by a straight (resp. dotted) box.}
	\label{fig:2circ}
\end{figure}

We easily check, as remarked beforehand, that if $\cH=\emptyset$ then the $\cI$ and $\cJ$ satisfy the assumptions of Corollary \ref{cor:plus}. Moreover, all the nonabelian cores arising in rank $4$ or less at very bad primes are heartless, except the $[3, 10, 9]$-core in $\rU\rF_4(2^f)$ which has been studied in \cite{HLM11} already; it is an immediate check that the $\cI$ and $\cJ$ previously defined satisfy the conditions of Corollary \ref{cor:plus} in this case as well. In the same fashion as \cite[Lemma 18]{LMP18b} (see the function \verb|findCircleZ| in \cite{LMPdF4}), we get the following result by a GAP4 implementation. 

\begin{lemma} \label{lem:corcov} Let $G$ be a finite Chevalley group of type $\rY$ and rank $r$, and let $\fC$ be a nonabelian core 
of $\rU\rY_r(p^f)$. If $r \le 4$, and if $p=2$ is a very bad prime for $G$, then the sets $\cI$ and $\cJ$ constructed as above satisfy the assumptions of Corollary \ref{cor:plus}.
\end{lemma}

We conclude this section with an equality that will be repeatedly used in the sequel. Let us take a nonabelian core $\fC$, and let $\cI=\{i_1, \dots, i_m\}$ and $\cJ=\{j_1, \dots, j_\ell\}$ be constructed as in Corollary \ref{cor:plus}. In order to determine $X'$ and $Y'$, we need to study the equation $\lambda([y, x])=1$ for $x=x_{i_1}(t_{i_1}) \cdots x_{i_m}(t_{i_m})\in X$ and $y=x_{j_1}(s_{j_1}) \cdots x_{j_\ell}(s_{j_\ell}) \in Y$. Its general form is 
\begin{equation} \label{eq:1i}
\phi(\sum_{h=1}^{\ell}\sum_{k=1}^{m} d_{h, k} s_{j_h}^{b_h} t_{i_k}^{c_k})=1,
\end{equation}
for some $d_{h, k} \in \F_q$ and $(b_h, c_k) \in (\Z_{\ge 1})^2$
. Hence we have 
$$
X'=\{x_{i_1}(t_{i_1}) \cdots x_{i_m}(t_{i_m})
\text{ such that Equation }\eqref{eq:1i} \text{ holds for all }
s_{j_1}, \dots, s_{j_\ell} \in \F_q^\ell
\},
$$
and 
$$
Y'=\{x_{j_1}(s_{j_1}) \cdots x_{j_\ell}(s_{j_\ell}) \text{ such that Equation }\eqref{eq:1i} \text{ holds for all } t_{i_1}, \dots, t_{i_m} \in \F_q^m
\}.
$$

\begin{remark}\label{rmk:heartl}
	Recall that all nonabelian cores in $\rU\rF_4(2^f)$, except the well-known $[3, 10, 9]$-core described in \cite[\S4.3]{GLMP16}, are heartless. That is, every index of a root in $\cS \setminus \cZ$ is involved in Equation \eqref{eq:1i}. Our focus for the rest of the work will be therefore on the determination of the solutions of Equation \eqref{eq:1i}, which is enough to completely determine a parametrization of $\Irr(X_\cS)_\cZ$ in the case of heartless cores. 
\end{remark}

\begin{remark}\label{rmk:multip}
	Although two $[z, m, c]$-cores are not always isomorphic, we can still group them by means of Equation \eqref{eq:1i}. Namely it is easy to see that two heartless cores for which Equation \eqref{eq:1i} is the same up to a permutation of indices determine the same numbers of irreducible characters and corresponding degrees. We will say in the sequel that such cores have the same \emph{branching}. 
%
\end{remark}

We collect in the third column of Table \ref{tab:parF4}, for a fixed form $[z, m, c]$, the number of cores in $\rU\rF_4(2^f)$ of that form that have the same branching.  In general, this considerably decreases the number of nonabelian cores to study. For example, we see from Table \ref{tab:parF4} that it is sufficient to study $14$ pairwise non-isomorphic nonabelian 
cores in type $\rF_4$ when $p=2$.




\section{The parametrizations of $\Irr(\rU\rF_4(2^{2k}))$ and $\Irr(\rU\rF_4(2^{2k+1}))$}\label{sec:comput}

As an application of the method previously developed, we give the parametrization of $\Irr(U)$ when 
$G=\rF_4(2^f)$. The labelling for the positive roots and the Chevalley relations are as in \cite[\S2.4]{GLMP16}. 
In this case, we have $190$ representable sets. The characters of abelian cores 
are immediately parametrized via the algorithmic procedure of Section \ref{sec:adaalg}; in fact, such 
characters had already been parametrized in \cite{Fal18}. 
We are left with examining $211$ nonabelian cores, 
which by Remarks \ref{rem:241} and \ref{rmk:multip} can in turn be reduced to the study of $13$ 
families only. We parametrize 
every character arising from nonabelian cores. 

\begin{theorem}
	All characters arising from nonabelian cores of $\rU\rF_4(2^f)$ are parametrized, and their branchings into $14$ families of $\Irr(\rU\rF_4(2^f))$ are listed in Table \ref{tab:parF4}. 
\end{theorem}

We now explain how to read Table \ref{tab:parF4}. 
The first column collects all the triples $[z, m, c]$ 
that arise as forms of nonabelian cores as in Section \ref{sec:adaalg}. The second column collects the 
number of occurrences of a core of a fixed form, and the third column describes their branching as explained in Remark \ref{rmk:multip}. 
%
Fixed a family $\cF=\Irr(X_\cS)_\cZ$, we gather in the fourth column 
of Table \ref{tab:parF4} the families $\cF^1, \dots, \cF^m$ where $m$ is the number of different branchings. 
The different labelling for each $1, \dots, m$ is reflected in the fifth column. This collects labels 
for an irreducible character of each family $\cF_i$ obtained as inflation/induction from an abelian subgroup of $X_\cS$, which is not necessarily a product of root subgroups and whose structure can be reconstructed by the indices of the labels. The convention for the letters $a$, $b$, $c$, $d$ for such labels and their precise meaning are explained in \cite[Section 5]{LMP18b}. Finally, we collect in the sixth column the number of irreducible characters of $\cF_i$, and in the seventh column their degree. 

The pathology of the case $\Irr(\rU\rF_4(q))$ when $q=2^f$ is quite rich. Notice that $f=2k$ if and only if $q \equiv 1 \mod 3$, 
and $f=2k+1$ if and only if $q \equiv -1 \mod 3$. For the first time in the study of any of the sets $\Irr(U)$, the parametrization is different according to the congruence class of $q$ modulo $3$. 
In fact, the families $\cF_{4, 2}$, $\cF_{9, 1}$ and $\cF_{11}$ yield different numbers 
of characters according to whether $f$ is even or odd. The expression of $|\cF|$ when $\cF$ is one of these families is \emph{not} polynomial in $q$, but PORC (Polynomial On Residue Classes) in $q$. Surprisingly, the global numbers $k(U, D)$ of irreducible characters of $\Irr(\rU\rF_4(q))$ of fixed degree $D$ are the same for every $D$ in both cases of $f$ odd and $f$ even. As remarked in the Introduction, an interesting research problem is to 
find an insightful explanation of this phenomenon. 

The number of irreducible characters of a fixed degree are collected in Table \ref{tab:fam2F4}. In particular, the degrees of characters in $\Irr(U)$ are: 
$q^i$ for $i=0, \dots, 9$; 
$q^i/2$ for $i=1, \dots, 10$; 
$q^i/4$ for $i \in \{4, 10\}$; and 
$q^4/8$. 
This is the example of smallest rank that yields a 
character of $\Irr(U)$ of degree $q^i/p^3$ when $q=p^f$. 

Finally, we point out that the analogue over bad primes of \cite[Conjecture B]{Is07} which generalizes 
\cite[Conjecture 6.3]{Leh74} 
does \emph{not} hold for the group $\rU\rF_4(2^f)$. In fact, the number 
$k(U, q^k)$ cannot always be expressed as a polynomial 
in $v:=q-1$ with non-negative integral coefficients. Moreover, $k(\rU\rF_4(q), q^4) ,  k(\rU\rF_4(q), q^4/4)\in\mathbb{Z}[v/3] \setminus \mathbb{Z}[v]$. A similar phenomenon happens when $p=3$ \cite[Table 3]{GLMP16}, in that $k(\rU\rF_4(q), q^4) ,  k(\rU\rF_4(q), q^4/3)\in\mathbb{Z}[v/2] \setminus \mathbb{Z}[v]$. If $p \ge 5$ then the expression of every  $k(\rU\rF_4(q), q^k)$ is in $\mathbb{Z}[v]$. 

%







\begin{table}[t]
\begin{tabular}{|l|l|}
\hline
$D$ & $k(\mathrm{UF}_4(q), D)$  \\
\hline
\hline
 $1$ & $v^4+4v^3+6v^2+4v+1 $   \\
\hline
 $q/2$ & $4v^4+8v^3+4v^2 $   \\
\hline
 $q$ & $2v^5+8v^4+14v^3+12v^2+4v $   \\
\hline
 $q^2/2$ & $8v^4+16v^3+8v^2 $   \\
\hline
$q^2$ & $ 2v^6+12v^5+27v^4+30v^3+17v^2+4v $   \\
\hline
$q^3/2$ & $ 12v^4+24v^3+12v^2  $   \\
\hline
$q^3$ & $8v^5+28v^4+36v^3+20v^2+4v $   \\
\hline
$q^4/8$ & $  8v^4 $   \\
\hline
$q^4/4$ & $  8v^6/3+80v^5/3+98v^4/3 $   \\
\hline
$q^4/2$ & $  10v^6+60v^5+114v^4+80v^3+8v^2 $   \\
\hline
$q^4$ & $2v^8+16v^7+160v^6/3+280v^5/3+301v^4/3+68v^3+23v^2+2v $ \\
\hline
$q^5/2$ & $8v^5+24v^4+24v^3+8v^2   $   \\
\hline
$q^5$ & $2v^7+14v^6+38v^5+50v^4+34v^3+12v^2+2v $ \\
\hline
$q^6/2$ & $ 16v^5+40v^4+32v^3+8v^2  $   \\
\hline
$q^6$ & $2v^7+15v^6+40v^5+53v^4+36v^3+13v^2+2v$  \\
\hline
$q^7/2$ & $4v^6+24v^5+48v^4+40v^3+12v^2   $   \\
\hline
$q^7$ & $2v^6+10v^5+20v^4+20v^3+10v^2+2v  $   \\
\hline
$q^8/2$ & $ 8v^5+32v^4+32v^3+8v^2  $   \\
\hline
$q^8$ & $v^6+8v^5+18v^4+18v^3+7v^2$   \\
\hline
$q^9/2$ & $ 8v^5+28v^4+24v^3+4v^2  $   \\
\hline
$q^9$ & $2v^4+4v^3+2v^2$  \\
\hline
$q^{10}/4$ & $ 16v^4 $   \\
\hline
$q^{10}/2$ & $ 8v^3  $   \\
\hline
\hline 
\multicolumn{2}{|c|}{ $k(\mathrm{UF}_4(q))= 2v^8+20v^7+104v^6+362v^5+674v^4+552v^3+194v^2+24v+1$}\\
\hline
\end{tabular}
\caption{The numbers of irreducible characters of $\mathrm{UF}_4(q)$ of fixed degree for $q=2^f$, where $v=q-1$.}
\label{tab:fam2F4}
\end{table}


Except for the $[3, 10, 9]$-core, whose parametrization is as for $\cF_{8, 9, 10}^{even}$ in \cite[Table 2]{HLM11}, all the other cases in Table \ref{tab:parF4} correspond to heartless cores. Let $(\cS, \cZ)$ be one such core. We apply the method in Section \ref{sec:armleg} to find 
$\cI$ and $\cJ$; these are readily computed thanks to our implemented function \verb|findCircleZ| in GAP4 \cite{LMPdF4}. Then $X'$ and $Y'$ can 
be determined by means of the study of Equation \ref{eq:1i}. As in \cite[\S5.1]{LMP18b}, if $X'$ is 
an abelian subgroup then the characters in $\Irr(X_\cS)_\cZ$ are immediately parametrized by inflating over $\tY\ker(\lambda)$ and inducing to $X_\cS$. This is the case for all remaining families in Table \ref{tab:parF4} except $\cF_{7, 2}$ and $\cF_{8}$; hence the only computation we have to do in these cases is to solve Equation \eqref{eq:1i}. The remaining two families yield $|X'|=q^2$ and $X'$ is not a subgroup of $X_\cS$. The study of the family $\cF_{8}$ remains uncomplicated as the associated graph $\Gamma$ 
has in this case just three edges. 
The study of the family $\cF_{7, 2}$ presents more complications and will be examined in full details. 

We include in this work the complete study of three 
important families of characters arising from nonabelian cores, namely:
\begin{itemize}
\item the family $\cF_{4, 2}$ corresponding to a $[4,8,4]$-core, which provides the smallest example where the expression of the cardinality of a family $\cF_i$ is PORC, but not 
polynomial, 
\item the family $\cF_{7, 2}$ corresponding to a $[4,12,9]$-core, where $X'$ is not a subgroup, which 
presents a more intricate branching and contains characters of degree $q^4/8$, and 
\item the family $\cF_{11}$ corresponding to a $[6,10,4]$-core, whose study 
requires the determination of solutions of complete cubic equations over $\F_q$. 
\end{itemize}
The difficulty of the computations related to all other families in Table \ref{tab:parF4} is bounded by that of the three families described above. Full details in these 
cases can be found in \cite{LMPdF4}. 




Before we start, we recall the following notation. For any $q=p^f$ and $m \ge 1$, we define 
$$\F_{q, m}^\times:=\{x \in \F_q^\times \mid x=y^m \text{ for some } y \in \F_q^\times\}.$$ 
Notice that $\F_{q, m}^\times$ is a cyclic group. We focus on the set $\F_{q, 3}^\times$ when $q=2^f$. It is 
easy to check that 
if $f=2k+1$ then $\F_{q, 3}^\times=\F_q^\times$, while if $f=2k$ then $|\F_{q, 3}^\times|=(q-1)/3$. 

%
%
%
%

We first study a nonabelian $[4,8,4]$-core arising from the family $\cF_{4, 2}$ in Table \ref{tab:parF4}. In this case, we have

\begin{itemize}
\item $\cS=\{ \alpha_{2}, \alpha_{3}, \alpha_{5}, \alpha_{7}, \alpha_{8}, \alpha_{9}, \alpha_{10}, \alpha_{18}\}$,
\item $\cZ=\{\alpha_{8}, \alpha_{9}, \alpha_{10}, \alpha_{18}\}$,
\item $\cA =\{\alpha_{1}, \alpha_{4}\}$ and $  \cL = \{\alpha_{11}, \alpha_{16}\}$, 
\item $\cI=\{ \alpha_2, \alpha_5\}$ and $\cJ=\{\alpha_3, \alpha_7 \}$.

\end{itemize}

\begin{prop}
\label{core[4,8,4]}
The irreducible characters corresponding to the family $\cF_{4, 2}$ in $\Irr(\rU\rF_4(2^f))$ are parametrized as follows: 
\begin{itemize}
\item If $f=2k$, then 
$$\cF_{4, 2}=:\cF_{4, 2}^{f \text{ even }}=\cF_{4, 2}^{f \text{ even, }1}\sqcup \cF_{4, 2}^{f \text{ even, }2},$$ 
where 
\begin{itemize}
\item $\cF_{4, 2}^{f \text{ even, }1}$ consists of $2(q-1)^4/3$ irreducible characters of degree $q^2$, and
\item $\cF_{4, 2}^{f \text{ even, }2}$ consists of $16(q-1)^4/3$ irreducible characters of degree $q^2/4$. 
\end{itemize}
\item If $f=2k+1$, then $\cF_{4, 2}=:\cF_{4, 2}^{f \text{ odd }}$ 
consists of $4(q-1)^4$ irreducible characters of degree $q^2/2$. 
\end{itemize}
The labels of the characters in $\cF_{4, 2}^{f \text{ even, }1}$, $\cF_{4, 2}^{f \text{ even, }2}$ and in  $\cF_{4, 2}^{f \text{ odd }}$ are collected in Table \ref{tab:parF4}. 
\end{prop}

\begin{proof}
Here, Equation \eqref{eq:1i} has the form 
\begin{align*}
\phi(
s_3(a_9t_2s_3+a_8t_5)
+s_7(a_{18}t_5s_7+a_{10}t_2)
)=1. 
\end{align*}
Hence we have 
\begin{align*}
X'&=\{x_2(t_2)x_5(t_5) \mid t_2, t_5 \in \F_q, 
a_8^2t_5^2=a_9t_2 \text{ and }a_{10}^2t_2^2=a_{18}t_5
\}\\
&=\{x_2(t_2)x_5(t_5) \mid t_2, t_5 \in \F_q, 
t_2=a_8^2a_9^{-1}t_5^2\text{ and }
t_5^4=a_8^{-4}a_9^2a_{10}^{-2}a_{18}t_5
\}
\end{align*}
and 
\begin{align*}
Y'&=\{x_3(s_3)x_7(s_7) \mid s_3, s_7 \in \F_q, 
a_9s_3^2=a_{10}s_7\text{ and }
a_{18}s_7^2=a_8s_3
\}\\
&=\{x_3(s_3)x_7(s_7) \mid s_3, s_7 \in \F_q, 
s_7=a_9a_{10}^{-1}s_3^2\text{ and }
s_3^4=a_8a_9^{-2}a_{10}^{2}a_{18}^{-1}s_3
\}.
\end{align*}

%
%

Let us assume that $f=2k$. If $a_{18} \notin a_8a_9^{-2}a_{10}^{2}\F_{q, 3}^\times$, that is, for $2(q-1)/3$ choices of $a_{18}$ in $\F_q^\times$, then 
the quartic equations involved in the definitions of $X'$ and $Y'$ just have a trivial solution. In this case, we have $X'=1$ and $Y'=1$, 
and we get the family 
$\cF_{4, 2}^{f \text{ even}, 1}$ 
as in Table \ref{tab:parF4}.

If $a_{18} \in a_8a_9^{-2}a_{10}^{2}\F_{q, 3}^\times$, i.e. for $(q-1)/3$ choices of $a_{18}$ in $\F_q^\times$, then there are 
three distinct values $\omega_{8, 9, 10, 18; i}$ for $i=1, 2, 3$, such that 
$\omega_{8, 9, 10, 18;i}^3=a_8a_9^{-2}a_{10}^{2}a_{18}^{-1}$.  
In this case, we have 
$$
X'=\{1\} \cup \{x_2(a_8a_9^{-1}\omega_{8, 9, 10, 18;i}^{-2})x_5(a_8^{-1}\omega_{8, 9, 10, 18;i}^{-1}) \mid i \in [1, 3]\}\},$$ 
and 
$$Y'=\{1\} \cup \{x_3(\omega_{8, 9, 10, 18;i})x_7(a_9a_{10}^{-1}\omega_{8, 9, 10, 18;i}^2) \mid i \in [1, 3]\}.$$
We now observe that $X'$ and $Y'$ are each isomorphic to 
$C_2 \times C_2$. We get the family $\cF_{4, 2}^{f \text{ even}, 2}$ 
as in Table \ref{tab:parF4}.
By Equation \eqref{eq:allfam}, we readily check that 
$$
\cF_{4, 2}=:\cF_{4, 2}^{f \text{ even }}=\cF_{4, 2}^{f \text{ even, }1}\sqcup \cF_{4, 2}^{f \text{ even, }2}. 
$$
This proves the first claim of the proof. 

Let us now assume that $f=2k+1$. Let $\omega_{8, 9, 10, 18}$ be the unique cube root of 
$a_8a_9^{-2}a_{10}^{2}a_{18}^{-1}$. Then we get 
$$X'=\{1, x_2(a_8a_9^{-1}\omega_{8, 9, 10, 18}^{-2})x_5(a_8^{-1}\omega_{8, 9, 10, 18}^{-1})\}, 
\qquad 
Y'=\{1, x_3(\omega_{8, 9, 10, 18})x_7(a_9a_{10}^{-1}\omega_{8, 9, 10, 18}^2)\}.$$
Hence we obtain the family 
$\cF_{4, 2}^{f \text{ odd}}$ as in Table \ref{tab:parF4}, completing our proof. 
\end{proof}


We then move on to study the family $\cF_{7, 2}$ in Table \ref{tab:parF4}, corresponding to nonabelian cores of the form $[4, 12, 9]$. Here we have 
\begin{itemize}
\item $\cS=\{ \alpha_{1}, \alpha_{2}, \alpha_{3}, \alpha_{4}, \alpha_{5}, \alpha_{6}, \alpha_{7}, \alpha_{8}, \alpha_{9}, \alpha_{10}, \alpha_{11}, \alpha_{16}\}$,
\item $\cZ=\{\alpha_{8}, \alpha_{10}, \alpha_{11}, \alpha_{16}\}$,
\item $\cA = \cL = \varnothing$, 
\item $\cI=\{ \alpha_1, \alpha_2, \alpha_3, \alpha_{4}\}$ and $\cJ=\{\alpha_5, \alpha_6, \alpha_{7}, \alpha_{9} \}$.
\end{itemize}

\begin{prop}
\label{core[4,12,9]}
The irreducible characters corresponding to the family $\cF_{7, 2}$ in $\Irr(\rU\rF_4(2^f))$ are parametrized as follows: 
$$
\cF_{7, 2}=\bigsqcup_{i=1}^8 \cF_{7, 2}^i,
$$
where
\begin{itemize}
\item $\cF_{7, 2}^1$ consists of $8(q-1)^4$ irreducible characters of degree $q^4/8$, and 
\item each of $\cF_{7, 2}^i$ for $i \in \{2, \dots, 8\}$ consists of $2(q-1)^4$ irreducible characters of degree $q^4/4$. 
\end{itemize}
The labels of the characters in $\cF_{7, 2}^i$ for $i \in \{1, \dots, 8\}$ are collected in Table \ref{tab:parF4}. 
\end{prop}

\begin{proof}
The form of Equation \eqref{eq:1i} is 
\begin{align*}
\phi(
s_5(a_{11}t_3^2+a_{8}t_{3})
+s_6(a_{8}t_1+a_{10}t_4)
+s_{7}(a_{16}t_2s_{7}+a_{10}t_{2})
+s_{9}(a_{16}t_4^2+a_{11}t_1)
)=1. 
\end{align*}
We have that
$$
X'=\{\ux(\ut) \in X \mid 
a_{11}t_3^2=a_{8}t_{3}, \,\,
a_{8}t_1=a_{10}t_4,\,\,
a_{10}^2t_{2}^2=a_{16}t_2\,\,\text{ and }\,\,
a_{16}t_4^2=a_{11}t_1
\}
$$
and
$$
Y'=\{\ux(\us) \in Y \mid 
a_8s_6=a_{11}s_9,\,\,
a_{16}s_7^2=a_{10}s_7,\,\,
a_8^2s_5^2=a_{11}s_5\,\,\text{ and }\,\,
a_{10}^2s_6^2=a_{16}s_9
\}.
$$

%

Hence we have that 
$$
X'=\left\{
x_1(c_1)
x_2(c_2)
x_3(c_3)
x_4\left(\frac{a_8}{a_{10}}c_1\right) \mid 
c_1\in \left\{0, \frac{a_{10}^2a_{11}}{a_8^2a_{16}}\right\}, 
c_2\in \left\{0, \frac{a_{16}}{a_{10}^2}\right\}, 
c_3\in \left\{0, \frac{a_8}{a_{11}}\right\} 
\right\}$$ 
and 
$$Y'=\left\{
x_5(c_1)
x_6(c_2)
x_7(c_3)
x_9\left(\frac{a_8}{a_{11}}c_2\right)
\mid 
c_1\in \left\{0, \frac{a_{11}}{a_8^2}\right\}, 
c_2\in \left\{0, \frac{a_8a_{16}}{a_{10}^2a_{11}}\right\}, 
c_3\in \left\{0, \frac{a_{10}}{a_{16}}\right\} 
\right\},$$
with $X'=X_1'X_2'X_3'$ in a natural way with 
$X_1' \subseteq X_2$, $X_2' \subseteq X_3$, and $X_3' \subseteq X_1X_4$ and 
$Y'=Y_1'Y_2'Y_3'$ with $Y_1' \subseteq X_7$, $Y_2' \subseteq X_5$, and $Y_3' \subseteq X_6X_9$. Now notice that 
$X'$ is \emph{not always} a group, namely we have
\begin{align*}
&\left[x_2\left(\frac{a_{16}}{a_{10}^2}\right), x_3\left(\frac{a_8}{a_{11}}\right)\right] =
x_6\left(\frac{a_8a_{16}}{a_{10}^2a_{11}}\right)x_9\left(\frac{a_8^2a_{16}}{a_{10}^2a_{11}^2}\right),\\
&\left[x_1\left(\frac{a_{10}^2a_{11}}{a_8^2a_{16}}\right)x_4\left(\frac{a_{10}a_{11}}{a_8a_{16}}\right), x_2\left(\frac{a_{16}}{a_{10}^2}\right)\right] =
x_5\left(\frac{a_{11}}{a_8^2}\right),\\
&\left[x_1\left(\frac{a_{10}^2a_{11}}{a_8^2a_{16}}\right)x_4\left(\frac{a_{10}a_{11}}{a_8a_{16}}\right), x_3\left(\frac{a_8}{a_{11}}\right)\right] =
x_7\left(\frac{a_{10}}{a_{16}}\right).
\end{align*}
Thus we have $[X_i', X_j']=Y_k'$ for every $i, j, k$ with $\{i, j, k\}=\{1, 2, 3\}$. 

For $c_1, c_2, c_3 \in \{0, 1\}$ we call $\lambda^{\uc}:=\lambda^{c_1, c_2, c_3}$ the extension of $\lambda$ to 
$X_{\cZ}Y'$ such that $\lambda^\uc(y_i)=y_i^{c_i}$ for every $y_i \in Y_i$ and $i=1, 2, 3$. An inflation and induction procedure from groups of order $q^4/8$ induces then a bijection 
$$
\Irr(X_\cS)_{\cZ} \longrightarrow \bigsqcup_{c_1, c_2, c_3 \in \{0, 1\}} \Irr(X'Y'Z \mid \lambda^\uc).
$$

Let us assume $c_i=1$ for every $i=1, 2, 3$. Then we can apply Proposition \ref{prop:redlem} with arm $X_1'$ and leg $X_2'$. In this case, we get a bijection 
$$
 \Irr(X'Y'Z \mid \lambda^{1, 1, 1}) \longrightarrow \Irr(X_3'Y_1'Y_2'Y_3'Z \mid \lambda^{1, 1, 1}
 ), 
$$
and $X_3'Y_1'Y_2'Y_3'Z$ is abelian. Hence we get the family $\cF_{7, 2}^8$ 
as in Table \ref{tab:parF4}. 

Let us now assume that $c_i =c_j=1$ and $c_k=0$ for any $\{i, j, k\}=\{1, 2, 3\}$. Proposition \ref{prop:redlem} applies here with arm $X_i'$ and leg $X_k'$. We have a bijection 
$$
 \Irr(X'Y'Z \mid \lambda^\uc) \longrightarrow \Irr(X_j'Y_1'Y_2'Y_3'Z \mid \lambda^\uc
 ), 
$$
with $X_j'Y_1'Y_2'Y_3'Z$ abelian. This gives the three families $\cF_{7, 2}^5$, $\cF_{7, 2}^6$ and $\cF_{7, 2}^7$ as in Table \ref{tab:parF4}. 

Let us then assume that $c_i =1$, and $c_j=c_k=0$ for any $\{i, j, k\}=\{1, 2, 3\}$. Proposition \ref{prop:redlem} now applies with arm $X_j'$ and leg $X_k'$. We have a bijection 
$$
 \Irr(X'Y'Z \mid \lambda^\uc) \longrightarrow \Irr(X_i'Y_1'Y_2'Y_3'Z \mid \lambda^\uc
 ), 
$$
with $X_k'Y_1'Y_2'Y_3'Z$ abelian. This gives the three families $\cF_{7, 2}^2$, $\cF_{7, 2}^3$ and $\cF_{7, 2}^4$ as in Table \ref{tab:parF4}. 

Finally, let us assume $c_1=c_2=c_3=0$. Then we have that 
$$
 \Irr(X'Y'Z \mid \lambda^{0, 0, 0}) \longrightarrow \Irr(X'Y'Z/Y' \mid \lambda^{0, 0, 0})
$$
is a bijection, and $X'Y'Z/Y'\cong X_1'X_2'X_3'ZY'/Y'$ is abelian. We have determined our family $\cF_{7, 2}^1$ of $8(q-1)^4$ irreducible characters 
of degree $q^4/8$ as in Table \ref{tab:parF4}. 

Equation \eqref{eq:allfam} now yields
$$
\cF_{7, 2}=\bigsqcup_{i=1}^8 \cF_{7, 2}^i,
$$ 
proving our claim. 
\end{proof}

We conclude our work by expanding the computations for the parametrization of the unique $[6,10,4]$-core, 
which corresponds to the family $\cF_{11}$ in Table \ref{tab:parF4}. As previously remarked, we need 
some properties of solutions of cubic equations in $\F_q$. 
For $a, b \in \F_q^\times$, let 
$$p_{a, b}(X):=X^3+aX+b.$$
Define the map $g:\F_q \to \F_q$ such that $g(x)=x^3+x$, and for $i \in \{0, 1, 3\}$ let us put
$$\cA_{i}:=\{(a, b) \in (\F_q^\times)^2 \mid p_{a, b}(X)=0 
\text{ has }i\text{ solutions in }\F_q\}.$$
By \cite[Equation (1.1)]{EY86} and the 
fact that 
$
(1, b) \in \cA_i$ implies $(a^2, a^3b) \in \cA_i$ for every $a \in \F_q^\times$, we have that 
\begin{itemize}
	\item $\cA_3=\{(a^2, a^3(x^3+x)) \mid a \in \F_q^\times, x \notin \{0, 1\}, 1+x^{-2} \in \im(g)\}$, 
	\item $\cA_1=\{(a^2, a^3(x^3+x)) \mid a \in \F_q^\times, x \notin \{0, 1\}, 1+x^{-2} \notin \im(g)\}$, and
	\item $\cA_0=(\F_q^\times)^2 \setminus (\cA_3 \cup \cA_1)$. 
\end{itemize}
In particular, we have 
$$ 
|\cA_3|=\frac{(q-1)(q-3+(-1)^{f+1})}{6}, \,\,\,
|\cA_1|=\frac{(q-1)(q-1+(-1)^{f+1})}{2}, \,\,\, 
|\cA_0|=\frac{(q-1)(q+(-1)^f)}{3}. 
$$

The next result follows directly by the explicit description of 
$\cA_i$ for $i \in \{0, 1, 3\}$ and a case-by-case discussion. We omit the lengthy, but straightforward proof. 

\begin{lemma}\label{lem:cubelec}
Let 
$$S=\{(b, c, t) \mid b \in \F_q^\times, \, t \in \F_q^\times \setminus \{b^3\} \text{ and } c \in \F_q^\times\setminus \{t\}\}, 
$$ 
and for every $(b, c, t) \in S$, let $p_{b, c, t}(X):=X^3 + (t/b+b^2)X + (t+c)$, and 
$$\cB_{i}:=\{(b, c, t) \in S \mid p_{b, c, t}(X)=0 
\text{ has }i\text{ solutions}\}.$$
Then we have that 
$$ 
|\cB_3|=\frac{(q-5)(q-3+(-1)^{f+1})}{6}, \,\,\, 
|\cB_1|=\frac{(q-3)(q-1+(-1)^{f})}{2}, \,\,\,
|\cB_0|=\frac{(q-2)(q+(-1)^{f+1})}{3}. 
$$
\end{lemma}

\begin{remark}
	The expressions of $\cB_3$, $\cB_1$ and $\cB_0$ in Lemma \ref{lem:cubelec} as polynomials in $q$ for even and odd $f$ are different. This is reflected in the sixth column of Table \ref{tab:parF4} for the $[6, 10, 4]$-core, and explains a difference in the parametrization of the family $\cF_{11} \subseteq \Irr(\rU\rF_4(2^f))$ in these two cases. 
\end{remark}

We return to the study of the family $\cF_{11}$. In this case, 

\begin{itemize}
\item $\cS=\{ \alpha_{2}, \alpha_{3}, \alpha_{5}, \alpha_{6}, \alpha_{7}, \alpha_{8}, \alpha_{9}, \alpha_{10}, \alpha_{12}, \alpha_{18}\}$,
\item $\cZ=\{\alpha_{6}, \alpha_{8}, \alpha_9, \alpha_{10}, \alpha_{12}, \alpha_{18}\}$,
\item $\cA =\{\alpha_{1}, \alpha_{4}\}$ and $ \cL = \{\alpha_{11}, \alpha_{16}\}$, 
\item $\cI=\{ \alpha_2, \alpha_{5}\}$ and $\cJ=\{ \alpha_{3}, \alpha_{7} \}$.
\end{itemize}

\begin{prop}
\label{core[6,10,4]}
The irreducible characters corresponding to the family $\cF_{11}$ in $\Irr(\rU\rF_4(2^f))$ are parametrized as follows: 
\begin{itemize}
\item If $f=2k$, then 
$$\cF_{11}=:\cF_{11}^{f \text{ even }}=\bigsqcup_{i=1}^7 \cF_{11}^{f \text{ even, }i},$$ 
where 
\begin{itemize}
\item $\cF_{11}^{f \text{ even, }1}$ consists of $(q-1)^4$ irreducible characters of degree $q^2$, 
\item $\cF_{11}^{f \text{ even, }2}$ consists of $4(q-1)^4(q-2)$ irreducible characters of degree $q^2/2$, 
\item $\cF_{11}^{f \text{ even, }3}$ consists of $2(q-1)^5/3$ irreducible characters of degree $q^2$, 
\item $\cF_{11}^{f \text{ even, }4}$ consists of $16(q-1)^4(q-4)/3$ irreducible characters of degree $q^2/4$, 
\item $\cF_{11}^{f \text{ even, }5}$ consists of $(q-1)^5(q-2)/3$ irreducible characters of degree $q^2$, 
\item $\cF_{11}^{f \text{ even, }6}$ consists of $2q(q-1)^4(q-3)$ irreducible characters of degree $q^2/2$, and 
\item $\cF_{11}^{f \text{ even, }7}$ consists of $8(q-1)^4(q-4)(q-5)/3$ irreducible characters of degree $q^2/4$. 
\end{itemize}

\item If $f=2k+1$, then 
$$\cF_{11}=:\cF_{11}^{f \text{ odd }}=\bigsqcup_{j=1}^6 \cF_{4, 2}^{f \text{ odd, }j},$$ 
where 
\begin{itemize}
\item $\cF_{11}^{f \text{ odd, }1}$ consists of $(q-1)^4$ irreducible characters of degree $q^2$, 
\item $\cF_{11}^{f \text{ odd, }2}$ and $\cF_{11}^{f \text{ odd, }3}$ consist of $4(q-1)^4(q-2)$ irreducible characters of degree $q^2/2$, 
\item $\cF_{11}^{f \text{ odd, }4}$ consists of $(q-1)^4(q-2)(q+1)/3$ irreducible characters of degree $q^2$, 
\item $\cF_{11}^{f \text{ odd, }5}$ consists of $2(q-1)^4(q-2)(q-3)$ irreducible characters of degree $q^2/2$, and  
\item $\cF_{11}^{f \text{ odd, }6}$ consists of $8(q-1)^4(q-2)(q-5)/3$ irreducible characters of degree $q^2/4$.  
\end{itemize}

\end{itemize}
The labels of the characters in $\cF_{11}^{f \text{ even, }i}$ for $i=1, \dots, 7$ and in $\cF_{1}^{f \text{ odd, }j}$ for $j=1, \dots, 6$ are collected in Table \ref{tab:parF4}. 
\end{prop}

\begin{proof}
The form of Equation \eqref{eq:1i} is 
\begin{align*}
\phi(
s_3(a_9t_2^2+a_6t_2+a_8t_5)
+s_7(a_{18}t_{5}^2+a_{12}t_5+a_{10}t_2))
=1.
\end{align*}
We have that 
\begin{align*}
X'=&\{\ux(\ut)\in X \mid a_8t_5=a_9t_2^2+a_6t_2\quad \text{and} \quad 
a_{10}t_2=a_{18}t_{5}^2+a_{12}t_5\}
\end{align*}
and
\begin{align*}
Y'=&\{\ux(\us)\in Y \mid a_6^2s_3^2+a_{10}^2s_7^2=a_9s_3\quad \text{and} \quad 
a_{12}^2s_7^2+a_8^2s_3^2=a_{18}s_7\}.
\end{align*}

We now focus on the determination of $X'$. Analogous computations can be carried out in order to determine $Y'$. We omit the details in the latter case, just mentioning that 
the cubic equations that show up in the study of $X'$ and 
$Y'$, which depend on $a_i$ for $i \in \{6, 8, 9, 10, 12, 18\}$, have the same number of solutions for each of the fixed values of the $a_i$'s in $\F_q^\times$. 

Let us fix $a_8$, $a_9$ and $a_{18}$ in $\F_q^\times$. By combining the equations defining $X'$, we substitute the value of $t_5$ as a function of $t_2$ into the first equation. Let us put $\ba_6:=a_6/a_9$, $\ba_{10}:=a_8^2a_{10}/(a_9^2a_{18})$ and $\ba_{12}:=a_6a_8a_{12}/(a_9^2a_{18})$. Then we get 
\begin{equation}\label{eq:t2}
t_2(t_2^3+(\ba_{12}/\ba_6+\ba_6^2)t_2+(\ba_{10}+\ba_{12}))=0.
\end{equation}
Since $X'$ is an abelian subgroup of $X_{\cS}$, and $Y'$ is determined in a similar way as previously remarked (in particular, $|X'|=|Y'|$), then each choice of $a_i$ for $i \in \{6, 10, 12\}$ such that Equation \eqref{eq:t2} has $k$ solutions yields $k^2(q-1)^3$ irreducible characters of degree $q^2/k$. The claim follows if we determine the number of solutions of Equation \eqref{eq:t2} for every $\ba_6,\ba_{10},\ba_{12} \in \F_q^\times$. 

Let us first assume that $\ba_{10}=\ba_{12}=\ba_{6}^3$; this happens for $q-1$ values of $\ba_6,\ba_{10},\ba_{12} \in \F_q^\times$. In this case, Equation \ref{eq:t2} is $t_2^4=0$ and just has the solution $t_2=0$. In this case, we get the family $\cF_{11}^{1}$ as in Table \ref{tab:parF4}. 

Let us then assume that $\ba_{12}\ne \ba_{6}^3$ and $a_{10}=a_{12}$; this happens for $(q-1)(q-2)$ values of $\ba_6,\ba_{10},\ba_{12} \in \F_q^\times$. In this case, 
Equation \ref{eq:t2} is $t_2^2(t_2^2+c)=0$, where $c=\ba_{10}+\ba_{12} \ne 0$, and we see that its two distinct solutions are $0$ and the unique square root of $c$. This gives the family $\cF_{11}^{2}$ as in Table \ref{tab:parF4}.

We now assume that $\ba_{12}= \ba_{6}^3$ and $a_{10}\ne a_{12}$; this happens for $(q-1)(q-2)$ values of $\ba_6,\ba_{10},\ba_{12} \in \F_q^\times$. Equation \ref{eq:t2} writes $t_2(t_2^3+d)=0$, where $d=\ba_{10}+\ba_6^3 \ne 0$. If $f=2k+1$, then $d$ has a unique cube root and the equation has two distinct solutions. This gives the family  $\cF_{11}^{f \text{ odd, }3}$ as in Table \ref{tab:parF4}. Let us then assume that $f=2k$. We distinguish two cases in turn. We first suppose that $\ba_{10} \in (\ba_6^3+\F_{q, 3}^\times)\setminus \{0\}=\ba_6^3+ \F_{q, 3}^\times\setminus \{\ba_6^3\}$; this happens for $(q-1)((q-1)/3-1)=(q-1)(q-4)/3$ values of $\ba_6,\ba_{10},\ba_{12} \in \F_q^\times$. In this case, $d$ has three distinct cube roots, and Equation \ref{eq:t2} has four distinct solutions. This gives the family $\cF_{11}^{f \text{ even, }4}$ as in Table \ref{tab:parF4}. Assume then that $\ba_{10} \in (\ba_6^3+\F_q\setminus \F_{q, 3}^\times)\setminus \{0\}=\ba_6^3+\F_q\setminus \F_{q, 3}^\times$; this happens for $2(q-1)^2/3$ values of $\ba_6,\ba_{10},\ba_{12} \in \F_q^\times$. In this case, $d$ has no cube roots. Therefore, Equation \ref{eq:t2} only has the solution $t_2=0$, which yields the family $\cF_{11}^{f \text{ even, }3}$ as in Table \ref{tab:parF4}. 

Finally, we assume that $\ba_{12}= \ba_{6}^3$ and $a_{10}\ne a_{12}$. Then we are in the assumptions of Lemma \ref{lem:cubelec} by setting $t=\ba_{12}$, $b=\ba_6$ and $c=\ba_{10}$. We readily get the families $\cF_{11}^{f \text{ even, }5}$ $\cF_{11}^{f \text{ even, }6}$ and $\cF_{11}^{f \text{ even, }7}$ as in Table \ref{tab:parF4} when $f=2k$, and the families $\cF_{11}^{f \text{ odd, }4}$ $\cF_{11}^{f \text{ odd, }5}$ and $\cF_{11}^{f \text{ odd, }6}$ as in Table \ref{tab:parF4} when $f=2k+1$, in the cases when the equation 
$$
t_2^3+(\ba_{12}/\ba_6+\ba_6^2)t_2+(\ba_{10}+\ba_{12})=0
$$
has $0$, $1$ or $3$ solutions respectively. 

Since
$$
\cF_{11}^{f \text{ even }}=\bigsqcup_{i=1}^7 \cF_{11}^{f \text{ even, }i}
\qquad
\text{and}
\qquad 
\cF_{11}^{f \text{ odd }}=\bigsqcup_{j=1}^6 \cF_{4, 2}^{f \text{ odd, }j},
$$
the claim is proved.
\end{proof}
%

\begin{small}
	\begin{center}
		\begin{table}[h]  \footnotesize 
			\begin{tabular}{|c|c|c|c|c|c|c|}
				\hline
				Form & Freq. & Branch. & Family & Label & Number & Degree \\ 
				\hline
				\hline
				$[ 2, 4, 1 ]$ & $185$ & $185$ & $\cF_1$ & $\chi_{c_2, c_3}^{a_6, a_9}$ & $4(q-1)^2$ & $q/2$  \\
				\hline
				\multirow{2}{*}{ $[ 3, 10, 9 ]$} &  \multirow{2}{*}{$1$}& \multirow{2}{*}{$1$} & $\cF_2^1$ & $\chi^{a_8, a_{12}, a_{13}}$ 
				& $(q-1)^3$ & $q^3$  \\
				\cdashline{4-7}
				&  &  & $\cF_2^2$ & $\chi_{c_{1, 3, 7}, c_2}^{a_{5, 6, 10}, a_8, a_{12}, a_{13}}$ & $4(q-1)^4$ & $q^3/2$ \\
				\hline
				$[ 4, 8, 2 ]$ &  $2$& $2$ & $\cF_3$ & $\chi_{c_2, c_3, c_5, c_7}^{a_6, a_9, a_{12}, a_{18}}$ 
				& $16(q-1)^4$ & $q^2/4$  \\
				\hline
				\multirow{4}{*}{$[ 4, 8, 4 ]$} &  \multirow{4}{*}{$8$}& $6$ & $\cF_{4, 1}$ 
				& $\chi_{c_{2,5}, c_{3,7}}^{a_6, a_8, a_{10}, a_{18}}$ & $4(q-1)^4$ & $q^2/2$  \\
				\cline{3-7}
				& &  \multirow{3}{*}{$2$} & $\cF_{4, 2}^{f \text{ even}, 1}$ & $\chi^{a_8, a_9, a_{10}, a_{18}^1}$ & $2(q-1)^4/3$ & $q^2$  \\
				\cdashline{4-7}
				& & & $\cF_{4, 2}^{f \text{ even}, 2}$ & $\chi_{d_{2, 5}, d_{3, 7}}^{a_8, a_9, a_{10}, a_{18}^2}$ & $16(q-1)^4/3$ 
				& $q^2/4$  \\
				\cdashline{4-7}
				& & & $\cF_{4, 2}^{f \text{ odd}}$ & $\chi_{c_{2,5}, c_{3, 7}}^{a_8, a_9, a_{10}, a_{18}}$ & $4(q-1)^4$ & $q^2/2$  \\
				\hline
				$[ 4, 10, 5 ]$ &  $2$& $2$ & $\cF_5$ & $\chi_{c_{1, 7}, c_{2, 6}, c_{4}, c_9}^{a_{5}, a_{8}, a_{13}, a_{16}}$
				& $16(q-1)^4$ & $q^3/4$  \\
				\hline
				$[ 4, 11, 6 ]$ &  $2$ & $2$ & $\cF_6$ & $\chi_{b_{4, 7, 12, 15}, c_{2, 6, 9}, c_{4, 7, 12, 15}}^{a_{10}, a_{16}, a_{19}, a_{24}}$
				& $4q(q-1)^4$ & $q^3/2$  \\
				\hline
				\multirow{9}{*}{$[ 4, 12, 9 ]$} &  \multirow{9}{*}{$2$}& $1$ & $\cF_{7, 1}$ & 
				$\chi_{b_{1, 5, 8, 11}, b_{4, 7, 10, 13}}^{a_{12}, a_{15}, a_{19}, a_{23}}$ & $q^2(q-1)^4$ & $q^3$  \\
				\cline{3-7}
				& & \multirow{8}{*}{$1$} & $\cF_{7, 2}^1$ & $\chi_{c_{1, 4}, c_2, c_3}^{a_8, a_{10}, a_{11}, a_{16}}$ & $8(q-1)^4$ & $q^4/8$  \\
				\cdashline{4-7}
				& & & $\cF_{7, 2}^2$ & $\chi_{c_{1, 4}}^{a_8, a_{10}, a_{11}, a_{16}, e_{6, 9}}$ & $2(q-1)^4$ & $q^4/4$  \\
				\cdashline{4-7}
				& & & $\cF_{7, 2}^3$ & $\chi_{c_2}^{a_8, a_{10}, a_{11}, a_{16}, e_7}$ & $2(q-1)^4$ & $q^4/4$  \\
				\cdashline{4-7}
				& & & $\cF_{7, 2}^4$ & $\chi_{c_3}^{a_8, a_{10}, a_{11}, a_{16}, e_5}$ & $2(q-1)^4$ & $q^4/4$  \\
				\cdashline{4-7}
				& & & $\cF_{7, 2}^5$ & $\chi_{c_{1, 4}}^{a_8, a_{10}, a_{11}, a_{16}, e_5, e_7}$ & $2(q-1)^4$ & $q^4/4$  \\
				\cdashline{4-7}
				& & & $\cF_{7, 2}^6$ & $\chi_{c_2}^{a_8, a_{10}, a_{11}, a_{16}, e_5, e_{6, 9}}$ & $2(q-1)^4$ & $q^4/4$  \\
				\cdashline{4-7}
				& & & $\cF_{7, 2}^7$ & $\chi_{c_3}^{a_8, a_{10}, a_{11}, a_{16}, e_{6, 9}, e_7}$ & $2(q-1)^4$ & $q^4/4$  \\
				\cdashline{4-7}
				& & & $\cF_{7, 2}^8$ & $\chi_{c_{1, 4}}^{a_8, a_{10}, a_{11}, a_{16}, e_5, e_{6, 9}, e_7}$ & $2(q-1)^4$ & $q^4/4$  \\
				\hline
				\multirow{2}{*}{$[ 5, 9, 3 ]$} &   \multirow{2}{*}{$2$}&  \multirow{2}{*}{$2$} & $\cF_{8}^1$ & 
				$\chi_{c_2, c_3, c_5, c_7}^{a_6, a_9, a_{10}^1, a_{12}, a_{18}} $ & $8(q-2)(q-1)^4$ & $q^2/4$  \\
				\cdashline{4-7}
				& & & $\cF_{8}^2$ & $\chi_{c_3, c_5}^{a_6, a_9, a_{10}^2, a_{12}, a_{18}} $ & $2q(q-1)^4$ & $q^2/2$  \\
				\hline
				\multirow{8}{*}{$[ 5, 9, 4 ]$} &  \multirow{8}{*}{$4$}& \multirow{6}{*}{$3$} & $\cF_{9, 1}^{f \text{ even},1}$ & 
				$\chi_{d_{2,5}, d_{3, 7}}^{a_6^3, a_8, a_9^3, a_{10}, a_{18}}$ & $8(q-1)^4(q-4)/3$ & $q^2/4$  \\
				\cdashline{4-7}
				& & & $\cF_{9, 1}^{f \text{ even}, 2}$ & 
				$\chi_{d_{2,5}, d_{3, 7}}^{a_6^1, a_8, a_9^1, a_{10}, a_{18}}$ & $2q(q-1)^4$ & $q^2/2$  \\
				\cdashline{4-7}
				& & & $\cF_{9, 1}^{f \text{ even}, 3}$ & 
				$\chi_{d_{2,5}, d_{3, 7}}^{a_6^0, a_8, a_9^0, a_{10}, a_{18}}$ & $(q-1)^5/3$ & $q^2$  \\
				\cdashline{4-7}
				& & & $\cF_{9, 1}^{f \text{ odd},1}$ & 
				$\chi_{d_{2,5}, d_{3, 7}}^{a_6^3, a_8, a_9^3, a_{10}, a_{18}}$ & $8(q-1)^4(q-2)/3$ & $q^2/4$  \\
				\cdashline{4-7}
				& & & $\cF_{9, 1}^{f \text{ odd},2}$ & 
				$\chi_{d_{2,5}, d_{3, 7}}^{a_6^1, a_8, a_9^1, a_{10}, a_{18}}$ & $2(q-1)^4(q-2)$ & $q^2/2$  \\
				\cdashline{4-7}
				& & & $\cF_{9, 1}^{f \text{ odd},3}$ & 
				$\chi_{d_{2,5}, d_{3, 7}}^{a_6^1, a_8, a_9^1, a_{10}, a_{18}}$ & $(q-1)^4(q+1)/3$ & $q^2$  \\
				\cline{3-7}
				& & \multirow{2}{*}{$1$} & $\cF_{9, 2}^1$ & $\chi^{a_6, a_8, a_{10}, a_{18}}$ & $(q-1)^4$ & $q^2$  \\
				\cdashline{4-7}
				& & & $\cF_{9, 2}^2$ & $\chi_{c_{2, 5}, c_{3, 7}}^{a_6, a_8, a_{10}, a_{12}^*, a_{18}}$ & $4(q-1)^4(q-2)$ & $q^2/2$  \\
				\hline
				\multirow{2}{*}{$[ 5, 11, 6 ]$} &   \multirow{2}{*}{$2$}&  \multirow{2}{*}{$2$} & $\cF_{10}^1$ 
				& $\chi_{c_{1, 3}, c_{4, 5}, c_{6}, c_{9}}^{a_7, a_8, a_{10}^1, 
					a_{14}, a_{16}}$ & $8(q-2)(q-1)^4$ & $q^3/4$  \\
				\cdashline{4-7}
				& & & $\cF_{10}^2$ & $\chi_{c_{1,3}, c_9}^{a_7, a_8, a_{10}^2, 
					a_{14}, a_{16}} $ & $2q(q-1)^4$ & $q^3/2$  \\
				\hline
				\multirow{11}{*}{$[ 6, 10, 4 ]$} & \multirow{11}{*}{$1$} & \multirow{11}{*}{$1$} & $\cF_{11}^1$ & $\chi^{a_9, a_{10}, a_{12}, a_{18}}$ & $(q-1)^4$ & $q^2$  \\
				\cdashline{4-7}
				&  &  & $\cF_{11}^2$ & $\chi_{b_{2, 5}, b_{3, 7}}^{a_8^*, a_9, a_{10}, a_{12}, a_{18}}$ & $4(q-1)^4(q-2)$ & $q^2/2$  \\
				\cdashline{4-7}
				&  &  & $\cF_{11}^{f \text{ even}, 3}$ & $\chi^{{a_6^*}^1, a_9, a_{10}, a_{12}, a_{18}}$ & $2(q-1)^5/3$ & $q^2$  \\
				\cdashline{4-7}
				&  &  & $\cF_{11}^{f \text{ even}, 4}$ & $\chi_{d_{2, 5}, d_{3, 7}}^{{a_6^*}^2, a_9, a_{10}, a_{12}, a_{18}}$ & $16(q-1)^4(q-4)/3$ & $q^2/4$  \\
				\cdashline{4-7}
				&  &  & $\cF_{11}^{f \text{ even}, 5}$ & $\chi^{a_6^0, a_8^0, a_9, a_{10}, a_{12}, a_{18}}$ & $(q-1)^5(q-2)/3$ & $q^2$  \\
				\cdashline{4-7}
				&  &  & $\cF_{11}^{f \text{ even}, 6}$ & $\chi_{c_{2, 5}, c_{3, 7}}^{a_6^1, a_8^1, a_9, a_{10}, a_{12}, a_{18}}$  & $2q(q-1)^4(q-3)$ & $q^2/2$  \\
				\cdashline{4-7}
				&  &  & $\cF_{11}^{f \text{ even}, 7}$ & $\chi_{d_{2, 5}, d_{3, 7}}^{a_6^3, a_8^3, a_9, a_{10}, a_{12}, a_{18}}$ & $8(q-1)^4(q-4)(q-5)/3$ & $q^2/4$  \\
				\cdashline{4-7}
				&  &  & $\cF_{11}^{f \text{ odd}, 3}$ & $\chi_{b_{2, 5}, b_{3, 7}}^{a_6^1, a_9, a_{10}, a_{12}, a_{18}}$ & $4(q-1)^4(q-2)$ & $q^2/2$  \\
				\cdashline{4-7}
				&  &  & $\cF_{11}^{f \text{ odd}, 4}$ & $\chi^{a_6^0, a_8^0, a_9, a_{10}, a_{12}, a_{18}}$ & $(q-1)^4(q-2)(q+1)/3$ & $q^2$  \\
				\cdashline{4-7}
				&  &  & $\cF_{11}^{f \text{ odd}, 5}$ & $\chi_{c_{2, 5}, c_{3, 7}}^{a_6^1, a_8^1, a_9, a_{10}, a_{12}, a_{18}}$ & $2(q-1)^4(q-2)(q-3)$ & $q^2/2$  \\
				\cdashline{4-7}
				&  &  & $\cF_{11}^{f \text{ odd}, 6}$ & $\chi_{d_{2, 5}, d_{3, 7}}^{a_6^3, a_8^3, a_9, a_{10}, a_{12}, a_{18}}$ & $8(q-1)^4(q-2)(q-5)/3$ & $q^2/4$  \\
				\hline
			\end{tabular}
			\medskip
			\caption{The irreducible characters of $\Irr(\rU\rF_4(2^f))$ parametrized by nonabelian cores. } 
			\label{tab:parF4}
		\end{table}
	\end{center}
\end{small}

\bibliographystyle{siam}
\bibliography{bibl}

\end{document}